\documentclass[10pt]{article}

\usepackage{times}
\usepackage{amsthm}
\usepackage{amsmath}
\usepackage{amssymb}
\usepackage{mathrsfs}
\usepackage{pb-diagram}
\usepackage[all]{xy}
\usepackage{color}
\usepackage{comment}

\usepackage{hyperref}

\renewcommand\eqref[1]{(\ref{#1})} 

\def\A{{\mathcal A}}

\def\N{\mathbb{N}}
\def\T{\mathbb{T}}

\def\V{\mathcal V}

\def\C{{\sf C}}

\def\1{{\boldsymbol 1}}
\def\B{{\mathbb B}}

\def\H{\mathcal H}

\def\R{\mathbb{R}}

\def\S{\mathcal S}

\def\U{\mathbb U}

\def\W{{\mathcal W}}
\def\di{\diamond}

\def\e{{\sf e}}
\def\g{{\mathfrak g}}
\def\m{{\sf m}}
\def\wm{\widehat{\sf m}}

\def\({\left(}
\def\[{\left[}
\def\){\right)}
\def\]{\right]}

\def\a{\mathfrak{a}}

\def\G{{\sf G}}
\def\wG{\widehat{\sf{G}}}
\def\p{\parallel}
\def\<{\langle}
\def\>{\rangle}

\def\Op{{\sf Op}}
\def\op{{\sf op}}
\def\Sch{{\sf Sch}}
\def\M{{\cal M}}
\def\c{{\sf c}}
\def\v{{\sf v}}

\def\V{{\sf V}}
\def\id{{\sf id}}

\def\fscr{\mathscr}
\def\a{{\sf a}}

\newtheorem{Theorem}{Theorem}[section]
\newtheorem{Remark}[Theorem]{Remark}
\newtheorem{Lemma}[Theorem]{Lemma}

\newtheorem{Corollary}[Theorem]{Corollary}
\newtheorem{Proposition}[Theorem]{Proposition}
\newtheorem{Definition}[Theorem]{Definition}
\newtheorem{Definition-Proposition}[Theorem]{Definition-Proposition}

\numberwithin{equation}{section}

\begin{document}


\title{Twisted Pseudo-differential Operators\\ on Type I Locally Compact Groups}

\date{\today}

\author{H. Bustos and M. M\u antoiu 
}
\date{\small}
\maketitle \vspace{-1cm}


\begin{abstract}
Let $\G$ be a locally compact group satisfying some technical requirements and $\wG$ its unitary dual. Using the theory of twisted crossed product $C^*$-algebras, we develop a twisted global quantization for symbols defined on $\G\times\wG$ and taking operator values. The emphasis is on the representation-theoretic aspect. For nilpotent Lie groups, the connection is made with a scalar quantization of the cotangent bundle $T^*(\G)$ and with a Quantum Mechanical theory of observables in the presence of variable magnetic fields.
\end{abstract}

\section{Introduction}\label{duci}

The aim of this article is to construct twisted pseudo-differential operators for symbols acting on  $\G\times \wG$\,, where $\G$ is a second countable type I locally compact group and $\wG$ is the unitary dual, formed of equivalence classes of irreducible representations of $\G$\,. Part of our motivation comes from the mathematical objects  involved in the quantization of a physical systems placed in a magnetic field. In \cite{MR}  a quantization for global operator-valued symbols was proposed and this type of calculus will be our model; we will generalize it to twisted pseudo-differential operators. The principal structures that arise from our constructions are the twisted crossed products of $C^{*}$-algebras and  twisted Weyl systems. Under some assumptions, both structures involve algebras of symbols and these algebras can be seen as a functional calculus associated to non-commutative observables.

The usual Wely calculus on the phase space $\R^{n}\times\R^{n}$ is connected with the crossed product $C^{*}$-algebra $\A\rtimes_{{\sf L}}\R^{n}$ associated to the action ${\sf L}$ of the group $\R^{n}$ by left translations over some $C^{*}$-algebra $\A$ composed of uniformly continuous functions on $\R^{n}$ which is invariant under translations. This approach has physical interpretation, modeling quantum observables suitably built from positions and momenta of a particle moving in $\R^{n}$. 

To include a variable magnetic field, one must introduce a ``twisted'' form of the Weyl calculus \cite{MP1,MPR1,MPR2}, obtaining a quantization of observables which is invariant under gauge transformations. This realization is based on twisting both the pseudo-differential calculus and the crossed products algebras by a $2$-cocycle defined on the group $\R^{n}$ with values in the unitary elements of the algebra $\A$\,. The 2-cocycle, that we denote by $\gamma$\,, is given by the imaginary exponential of the magnetic flux through suitable triangles. The $C^{*}$-algebraic approach permits to understand the magnetic pseudo-differential operators through a representation in $L^{2}(\R^{n})$ of a twisted crossed product. Indeed, it corresponds to the Schr\" odinger representation of the twisted dynamical system associated to the algebra $\A\!\rtimes_{{\sf L}}^{\gamma}\!\R^n$\,; this Schr\"odinger representation incorporates magnetic translations. Composing this representation with a partial Fourier transform, one obtains pseudo-differential operators on $\R^{n}\times\widehat{\R^{n}}$.

The Fourier transform plays a principal role in the construction of such quantizations. Harmonic Analysis tools can be used to generalize the construction, for Abelian groups for instance \cite{MPR1}, where the phase space $\R^{n}\times\widehat{\R^{n}}$ is replaced by the product of an Abelian group $\G$ with its unitary Pontryagin dual $\wG$\,, which is also an Abelian locally compact group. 

The present article generalizes the twisted pseudo-differential operators from the Abelian case \cite{MPR1} to non-commutative phase space. The main obstacle is the fact that, when $\G$ is not Abelian, the unitary dual $\wG$ is no longer a group. It still has a a natural (but complicated) measure theory with respect to which a non-commutative version of Plancherel theorem holds. The cohomological twisting will be shown to be possible in this context.

On the other hand, we also extend the highly non-commutative formalism introduced in \cite{MR} to the presence of a group $2$-cocycle. The constructions in \cite{MR} have been predated by the intensive study of pseudo-differential operators on particular types of non-commutative Lie groups, as compact or nilpotent. We cite the articles \cite{FR,FR2,RT1,RTW} and the books \cite{FR1,RT}; they contain many other relevant references. 

The present setting of second countable unimodular type I groups is quite general (besides compact and nilpotent, it contains the Abelian, exponentially solvable or semisimple groups, motion groups and certain discrete groups). In particular, some of them do not possess a Lie structure. Hopefully, in some subsequent publication, we will restrict to smaller classes of groups allowing a deeper analytical investigation, involving more realistic spaces of functions. For $\G=\R^n$ and for magnetic-type cocycles this has been undertaken in \cite{IMP1,IMP2} and spectral results for magnetic Hamiltonians are contained in \cite{MPR1}.

Let us  describe the content. We start exposing briefly the $C^{*}$-algebraic formalism of  twisted crossed products and its representation theory. The standard version can be found in \cite{BS,PR1}. Actually we present a modified version, containing a parameter $\tau$ connected to ordering issues. Then we recall some basics things about the cohomology of groups $\G$ with coefficients in an Abelian Polish group $\mathbb U$\,. The most interesting case is $\mathbb U=C(\G,\mathbb{T})$, the group of continuous functions on the locally compact group $\G$ with values in the torus; it is related with the unitary multipliers of the $C^{*}$-algebras we deal with. We describe $2$-cocycles and their pseudo-trivializations which are related, in the case of $\R^{n}$, with vector potentials corresponding to a given magnetic field.

We go on with concrete realizations of the crossed products. We consider $C^{*}$-algebras composed of functions which are bounded and uniformly continuous over $\G$ and invariant under translations. Compatible data are defined as couples $(\A,\gamma)$ including a $C^{*}$-algebra $\A$ and a 2-cocycle $\gamma$ of the group $\G$ with values in the unitary of the multipliers of $\A$ such that the triple $(\A,\G,\gamma)$ is a twisted dynamical system.  As mentioned above, a  map $\tau:\G\to\G$ is assumed given, leading to the so called $\tau$-quantizations. For $\tau\equiv {\sf e}$ one obtains the analog of the Kohn-Nirenberg quantization and, for simplicity, the reader could only consider this case. Symmetric Weyl forms are available on certain groups and for certain type of cocycles. We also study the Sch\"odinger representations of the twisted dynamical system, related to the pseudo-trivializations of the cocycle.

In the third section we introduce the twisted pseudo-differential calculus. We begin with notions and results about the Fourier transform on type I groups, referring mainly to \cite{Di}. With the Sch\"odinger representation introduced previously and a partial Fourier transform we define the twisted pseudo-differential operators and examine their integral kernels. By suitable particularizations one gets multiplication as well as twisted convolution operators. The formalism is based on a Fourier-Wigner transform and a twisted Weyl system, involving unitary operators indexed by the points $(x,\xi)\in\G\times\wG$ but acting in different Hilbert spaces $L^2(\G)\otimes\H_\xi$ connected to the irreducible representations. They are built out of two simpler non-commuting families of unitary operators: one of them is associated to "position", while the other generalizes the magnetic translations. For $\R^{n}$ these families correspond to the unitary groups defined by position and momentum operators. We finish this section discussing products formulas and the existence of symmetric (Weyl) twisted quantization, for which the symbol of the adjoint operator is very simple.

In the final section we focus on the case of connected simply connected nilpotent Lie groups. Using the Lie algebra and properties of the exponential map, we show that the quantization can also be realized by scalar symbols defined globally on the cotangent bundle $T(\G)$\,. We also study cocycles which are defined by 2-forms on the Lie algebra, seen as magnetic fields on the Lie group. Various previous results showing the independence of the calculus on the chosen pseudo-trivialization of the $2$-cocycle can be seen here as a gauge invariance principle.

\section{The twisted crossed product formalism}\label{traci}

Let us fix some of our notations. For any (complex, separable) Hilbert space $\H$ one denotes by $\mathbb B(\H)$ the $C^*$-algebra of all linear bounded operators in $\H$\,, by $\mathbb B^2(\H)$ the $^*$-ideal of all Hilbert-Schmidt operators and by $\mathbb K(\H)$ the closed bi-sided $^*$-ideal of all the compact operators. The unitary operators form a group $\mathbb U(\H)$\,.
Let $\G$ be a locally compact second countable unimodular group with unit $\e$\,, fixed left Haar measure $\m$ and associated Lebesgue spaces $L^p(\G)\equiv L^p(\G;\m)$\,.
For a $C^{*}$-algebra $\A$ we denote by ${\sf Aut}(\A)$ the group of all its automorphisms with the strong topology. We also denote by $\M(\A)$ the multiplier algebra and by $\U[\mathcal M(\A)]\equiv\U(\A)$ its group of unitary elements. 

\subsection{$\tau$-twisted crossed products}\label{draci}

For the general theory of twisted $C^*$-dynamical systems and twisted crossed product algebras we refer to \cite{BS,PR1,PR2}. The untwisted case ($\gamma(\cdot,\cdot)=1$) is studied deeply in \cite{Wi}. We will need a modification of the usual definitions, to accommodate later $\tau$-quantizations related to ordering issues; {\it $\tau:\G\to\G$ will always be a continuous map, which does not need to be a group morphism or to commute with inversion}. We shall only consider Abelian $C^{*}$-algebras.

\begin{Definition}\label{bazadata}
We call {\rm twisted $C^*$-dynamical system} a quadruple $(\A,\a,\gamma,\G)$\,, where $\a:\G\to{\sf Aut}(\A)$ is a  strongly continuous action of $\,\G$ in $\A$ and $\gamma:\G\times\G\to\U(\A)$ is a strictly continuous map satisfying for $x\in\G$ the normalization condition $\gamma(x,\e)=1_{\A}=\gamma(\e,x)$
and for every $x,y,z\in\G$ the $2$-cocycle identity
\begin{equation}\label{comit}
\gamma(x,y)\,\gamma(xy,z)=\a_x[\gamma(y,z)]\,\gamma(x,yz)\,.
\end{equation}
\end{Definition}

In the next definition, the case $\tau(\cdot)=\e$ is the one appearing in the literature; the untwisted case including the parameter $\tau$ can be found in \cite[Subsect. 7.1]{MR}.  The reader only interested in the Kohn-Nirenberg version of pseudo-differential operators can stick to this situation; the formulas will look simpler. 

\begin{Definition}
\begin{enumerate}
\item[(i)]
To the twisted $C^*$-dynamical system $(\A,\a,\gamma,\G)$ we associate the Banach $^*$-algebra $L^1(\G;\A)$ (the space of Bochner integrable functions $\G\to\A$ with the obvious norm) with operations
\begin{equation*}\label{diamond}
\big(\Phi\diamond^\tau_\gamma \Psi\big)(x):=\int_{\G}\a_{\tau(x)^{-1}\tau(y)}[\Phi(y)]\,\a_{\tau(x)^{-1}y\tau(y^{-1}x)}\!\left[\Psi(y^{-1}x)\right]\a_{\tau(x)^{-1}}\big[\gamma(y,y^{-1}x)\big]d\m(y)\,,
\end{equation*}
\begin{equation*}\label{willibrant}
\Phi^{\diamond^\tau_\gamma}(x):=\a_{\tau(x)^{-1}}\big[\gamma\big(x,x^{-1}\big)\big]^*\,\a_{\tau(x)^{-1}x\tau(x^{-1})}\!\left[\Phi(x^{-1})^*\right].
\end{equation*}
\item[(ii)]
{\rm The crossed product $C^*$-algebra} $\A\rtimes^\tau_{\a,\gamma}\G:={\rm Env}\!\left[L^1(\G;\A)\right]$ is the enveloping $C^*$-algebra of this Banach $^*$-algebra, i.e its completion in the universal norm $\,\p\!\Phi\!\p_{\rm univ}\,:=\,\sup_{\Pi}\p\!\Pi(\Phi)\!\p_{\mathbb B(\H)}$\,,
where the supremum is taken over all the $^*$-representations $\,\Pi:L^1(\G,\A)\rightarrow\mathbb B(\H)$\,. 
\end{enumerate}
\end{Definition}

The Banach space $L^1(\G;\A)$ can be identified with the projective tensor product $\A\,\overline\otimes\,L^1(\G)$\,, and $\mathcal C_{\rm c}(\G;\A)$\,, the space of all $\A$-valued continuous compactly supported function on $\G$\,, is a dense $^*$-subalgebra. 

\begin{Definition}\label{peter}
Let $(\A,\a,\gamma,\G)$ be a twisted $C^*$-dynamical system. 
\begin{enumerate} 
\item[(i)]
{\rm A covariant representation} is a triple $(\rho,T,\H)$ where $\H$ is a Hilbert space, $\rho:\A\rightarrow\mathbb B(\H)$ is a $^*$-representation,
$T:\G\rightarrow\mathbb U(\H)$ is strongly continuous and satisfies
\begin{equation*}\label{clasa}
T(x)T(y)=\rho[\gamma(x,y)]T(xy)\,,\quad\forall\,x,y\in\G\,,
\end{equation*}
\begin{equation*}\label{holds}
T(x)\rho(a)T(x)^*=\rho\left[\a_x(a)\right]\,,\quad\forall\,a\in\A\,,\,x\in\G\,.
\end{equation*}
\item[(ii)]
{\rm The integrated form} of the covariant representation $(\rho,T,\H)$ is the unique continuous extension $\,\rho\!\rtimes^\tau\!T:\A\!\rtimes_{\a,\gamma}^\tau\!\G\rightarrow\mathbb B(\H)$ of the map defined on $L^1(\G;\A)$ by
\begin{equation*}\label{maximilian}
\big(\rho\!\rtimes^\tau\!T\big)(\Phi):=\!\int_\G\rho\big[\a_{\tau(x)}\big(\Phi(x)\big)\big]T(x)d\m(x)\,.
\end{equation*}
\end{enumerate}
\end{Definition}
 
The covariant representations are in 1-1 correspondence  with representations of the crossed product, one direction being given by the formula above for $\rho\!\rtimes^\tau\!T$\,.

\begin{Remark}\label{tauiso}
{\rm For different maps $\tau$ and $\tau'$ the algebras $\A\!\rtimes_{\a,\gamma}^{\tau'}\!\G$ and $\A\!\rtimes_{\a,\gamma}^{\tau}\!\G$ are isomorphic.  The isomorphism at the level of $L^{1}(\G,\A)$ is given by 
$$\big[\Theta_{\tau,\tau'}(\Phi)\big](x):=\a_{\tau(x)^{-1}\tau'(x)}\[\Phi(x)\] $$  
and it satisfies $\rho\!\rtimes^{\tau'}\!T=\big(\rho\!\rtimes^\tau\!T\big)\circ\Theta_{\tau,\tau'}$ for any covariant representation $(\rho,T,\H)$\,.
}
\end{Remark}

\subsection{Group cohomology}\label{flekarin}

We recall some definitions in group cohomology. Let $\G$ be a locally compact group and ${\cal U}$ a Polish (metrizable separable and complete) Abelian group. We suppose that there exists a continuous action $\a$ of $\G$  by automorphism of ${\cal U}$\,, thus ${\cal U} $ is a right $\G$-module. 

For $n\in\N$, the family of all continuous functions $\G^{n}\to{\cal U}$ is denoted by $C^{n}(\G;{\cal U})$\,. It is an Abelian group with the pointwise multiplication. The elements of  $C^{n}(\G;{\cal U})$ are called {\it $n$-cochains}. Let us define the {\it $n$-coboundary map} $\,\delta^{n}:C^{n}(\G;{\cal U})\to C^{n+1}(\G;{\cal U})$ by 
$$
(\delta^{n}\nu)(x_{1},\cdots,x_{n+1}):= \a_{x_{1}}[\nu(x_{2},\cdots,x_{n+1})]\prod_{j=1}^{n}\nu(x_{1},\cdots,x_{j}x_{j+1},\cdots,x_{n+1})^{(-1)^{j}}\!\nu(x_{1},\cdots,x_{n})^{(-1)^{n+1}}\!.
$$

A straightforward computation shows that $\delta^{n}$ is a group morphism and one has $\delta^{n+1} \circ\delta^{n}=0$\,. 

\begin{Definition}
\begin{enumerate}
  \item[(i)] $Z^{n}(\G;{\cal U}):=\ker(\delta^{n})$ is called the group of \,{\rm $n$-cocycles}.
  \item[(ii)] $B^{n}(\G;{\cal U}):={\rm im}(\delta^{n-1})$ is called the group of \,{\rm $n$-coboundaries}.
  \item[(iii)] The quotient $H^{n}(\G;{\cal U}):=	Z^{n}(\G;{\cal U})/B^{n}(\G;{\cal U})$ is called {\rm the $n$'th cohomology group of $\,\G$ with coefficients in ${\cal U}$} and its elements are {\rm called classes of cohomology}.
\end{enumerate}
\end{Definition}

We review the cases $n=0,1,2$\,. For $n=0$ one has $C^{0}(\G;{\cal U}):={\cal U}$ and 
$$
[\delta^{0}(a)](x)=\a_{x}(a)a^{-1},\quad\forall\,a\in{\cal U}\,,\,x\in\G\,.
$$ 
Thus $Z^{0}(\G;{\cal U})={\cal U}^{\G}$, the set of all fixed points of $\a$\,. By convention $B^{0}(\G;{\cal U})=\{1\}$\,.

The map $\delta^{1}:C^{1}(\G;{\cal U})\to C^{2}(\G;{\cal U})$ corresponds to 
$$
\big[\delta^{1}(\beta)\big](x,y)= \a_{x}[\beta(y)]\beta(x)\beta(xy)^{-1},
$$ 
thus a 1-cochain is a $1$-cocycle satisfying $\beta(xy)=\a_{x}[\beta(y)]\beta(x)$ for any $x,y\in\G$\,. Then $B^{1}(\G;{\cal U})$ consists of elements of the form $\a_{x}(a)a^{-1}$ for $a\in{\cal U}$\,.

For $n=2$ the coboundary map is 
$$
[\delta^{2}(\gamma)](x,y,z)=\a_{x}[\gamma(y,z)]\gamma(xy,z)^{-1}\gamma(x,yz)\gamma(x,y)^{-1}.
$$
Thus the 2-cocycles are exactly the 2-cochains which satisfy the identity (\ref{comit}) and $B^{1}(\G;{\cal U})$ is the set of elements of the form $\a_{x}[\beta(y)]\beta(x)\beta(xy)^{-1}$ for some 1-cochain $\beta$\,.

\begin{Remark}
{\rm
For two $C^{*}$-dynamical systems $(\A,\a,\gamma,\G)$ and $(\A,\a,\gamma',\G)$, one can reinterpret $\gamma,\,\gamma'$  as elements of $Z^{2}\big(\G;\U(\A)\big)$\,. If these cocycles are cohomologous, i.e. $\gamma'=\delta^{1}(\beta)\gamma$ for some 1-cochain $\beta$\,, the $C^{*}$ -algebras $\A\rtimes^{\tau}_{\a,\gamma}\!\G$ and $\A\rtimes^{\tau}_{\a,\gamma'}\!\G$ are isomorphic. To show this, by Remark \ref{tauiso}, is enough to consider the case $\tau\equiv {\sf e}$\,; then one can define the isomorphism on $L^{1}(\G,\A) $ by
$$
\big[\Upsilon_\beta(\Phi)\big](x)=\Phi(x)\beta(x)\,.
$$ 
}
\end{Remark}

Special attention deserves the case $\,{\cal U }=C(\G;\mathbb{T})$\,. This group, with the  topology of uniform convergence on compact sets, is the unitary part of the multiplier algebra of the $C^{*}$-algebra $C_{0}(\G)$ and will play an important role in the next subsection. The action consists of left translations by elements of $\G$ in  $C(\G,\mathbb{T})$\,.

\begin{Proposition}\label{cohomology}
For $n\ge 1$ , $H^{n}\big(\G;C(\G,\mathbb{T})\big)=\{1\}$\,.
\end{Proposition}

\begin{proof}
Any $\,\nu^{n}\in Z^{n}\big(\G;C(\G,\mathbb{T})\big)$ satisfies the identity
$$
\a_{x_{1}}[\nu^{n}(x_{2},\cdots,x_{n+1})]\prod_{j=1}^{n}
\nu^{n}(x_{1},\cdots,x_{j}x_{j+1},\cdots,x_{n+1})^{(-1)^{j}} \nu^{n}(x_{1},\cdots,x_{n})^{(-1)^{n+1}}=1
$$
for all $x_{1},\cdots,x_{n+1}\in\G$\,, which can be transformed into
$$
\a_{x_{1}}[\nu^{n}(x_{2},\cdots,x_{n+1})]=\nu^{n}(x_{1}x_{2},\cdots,x_{n+1})\prod_{j=1}^{n-1} \nu^{n}(x_{1},\cdots,x_{j}x_{j+1},\cdots,x_{n})^{(-1)^{j}}\nu^{n}(x_{1},\cdots, x_{n})^{(-1)^{n}}.
$$
Changing $x_1$ into $x_1^{-1}$ and evaluating at $x={\sf e}$ in both sides of the equality one gets
$$
\begin{aligned}
\big[\nu^n(x_{2},\cdots,x_{n+1})\big](x_{1})&=\big[\nu^n(x_{1}^{-1}x_{2},\cdots,x_{n+1})\big]({\sf e})\\
&\cdot\prod_{j=1}^{n-1} \big[\nu^{n}(x_{1}^{-1},\cdots,x_{j},x_{j+1},\cdots,x_{n})^{(-1)^{j}}\big]({\sf e})\big[\nu^{n}(x_{1}^{-1},\cdots, x_{n})^{(-1)^{n}}\big]\!({\sf e})\,,
\end{aligned}
$$
since $\a$ consists of left translations. The previous equation means that $\,\nu^{n}=\delta^{n-1}(\nu^{n-1})$\,, where 
\begin{equation}\label{miss}
\big[\nu^{n-1}(z_{1},\cdots,z_{n-1})\big](x)=\big[\nu^{n}(x^{-1},z_{1},\cdots,z_{n-1})\big]({\sf e})\,.
\end{equation}
This shows that any cocycle is a coboundary algebraically. 

We also need to prove that $\nu^{n-1}\in C^{n-1}(\G;C(\G;\mathbb{T}))$\,. This follows easily from the identification between $C^{n-1}(\G;C(\G;\mathbb{T}))$ and $C(\G\times\G^{n-1};\mathbb{T})$ and the explicit definition of the map $\nu^{n-1}$.
\end{proof}

The situation described in the next definition will be relevant below.

\begin{Definition}\label{grostie}
Let ${\cal U}$ be an Abelian Polish $\G$-module with action $\a$\,. An element $\gamma\in Z^{2}(\G;{\cal U})$ is said to be  {\rm pseudo-trivializable} if there is a $\G$-module ${\cal U}'$ with action $\a'$ such that ${\cal U}$ is a subgroup of ${\cal U}'$, $\a'_x|_{{\cal U}}=\a_x$ for every $x\in\G$ and $\gamma\in B^{2}(\G;{\cal U}')$\,.
\end{Definition}

\subsection{Concrete twisted crossed product $C^*$-algebras}\label{flegarin}

Let $\G$ a locally compact unimodular group.   By ${\sf Aut}(\G)$ we denote the group of all its (continuous) automorphisms. Let us denote the left and right actions ${\sf l},{\sf r}:\G\to{\sf Aut}(\G)$ of $\G$ on itself by
\begin{equation}\label{lef}
{\sf l}_y(x):=y^{-1}x\,,\quad{\sf r}_y(x):=xy\,.
\end{equation}
They commute: ${\sf l}_y{\sf r}_z{=\sf r}_z{\sf l}_y$ for every $y,z\in\G$\,, and are transitive. In addition, they induce actions on the $C^*$-algebra $\mathcal C_b(\G)$ of all bounded continuous complex functions on $\G$\,. To get pointwise continuous actions we restrict, respectively, to bounded left or right uniformly continuous functions:
\begin{equation}\label{lefd}
{\sf L}:\G\to{\sf Aut}\big[\mathcal{LUC}(\G)\big]\,,\quad\big[{\sf L}_y(c)\big](x):=\big(c\circ{\sf l}_y\big)(x)=c(y^{-1}x)\,,
\end{equation}
\begin{equation}\label{righd}
{\sf R}:\G\to{\sf Aut}\big[\mathcal{RUC}(\G)\big]\,,\quad\big[{\sf R}_y(c)\big](x):=\big(c\circ{\sf r}_y\big)(x)=c(xy)\,.
\end{equation}

A $C^*$-subalgebra $\A$ of $\mathcal{LUC}(\G)$ is called {\it left-invariant} if ${\sf L}_y\A\subset\A$ for every $y\in\G$\,; {\it right-invariance} is defined analogously for $C^*$-subalgebras of $\mathcal{RUC}(\G)$\,. Then $C_0(\G)$\,, the $C^*$-algebra of continuous complex functions which decay at infinity (arbitrarily small outside arbitrarily large compact subsets), is an invariant ideal, both to the left and to the right. We also denote by $C_{c}(\G)$ the *-algebra of complex continuous compactly supported functions.

\begin{Definition}\label{barzadata}
We call {\rm basic data} a pair $(\A,\gamma)$\,, where $\A$ is a left-invariant $C^*$-subalgebra of $\mathcal{LUC}(\G)$\,,
$\gamma:\G\times\G\times\G\to\mathbb T$ is a continuous map satisfying for every $x,q\in\G$ the normalization condition
\begin{equation*}\label{normaz}
\gamma(q;x,\e)=1=\gamma(q;\e,x)
\end{equation*}
and for every $x,y,z,q\in\G$ the $2$-cocycle identity
\begin{equation}\label{comits}
\gamma(q;x,y)\,\gamma(q;xy,z)=\gamma(x^{-1}q;y,z)\,\gamma(q;x,yz)\,.
\end{equation}
\end{Definition}

Note that $C(\G;\T)$ is the unitary group of the $C^*$-algebra $\mathcal{LUC}(\G)$\,. Thus  one can regard $\gamma$ as a function 
\begin{equation}\label{reinterp}
\gamma:\G\times\G\to C(\G;\T)\equiv\U[\mathcal{LUC}(\G)]\,,\quad \gamma(q;x,y):= [\gamma(x,y)](q)\,.
\end{equation}
We say that $(\A,\gamma)$ is {\it a compatible basic data} if $\gamma(x,y)\in\U(\A)\,,\ \forall\,x,y\in\G$\,.
Then \eqref{comits} can be rewritten as $\delta^2(\gamma)=1$ and $(\A,{\sf L},\gamma,\G)$ will be called {\it a concrete twisted $C^*$-dynamical system}. 

For $\A$-valued functions $\Phi$ defined on $\G$ and for elements $x,q$ of the group, we are going to use notations as $[\Phi(x)](q)=\Phi(q;x)$\,, interpreting $\Phi$ as a function of two variables.  We rewrite the general formulas defining the $\tau$-twisted crossed product in a more concrete form:
\begin{equation*}\label{ulrich}
\p\!\Phi\!\p_{(1)}\,=\int_{\G}\underset{q\in\G}{\rm ess\ sup}|\Phi(q;x)|\,d\m(x) \,,
\end{equation*}
\begin{equation*}\label{tiamond}
(\Phi\diamond^{\tau}_{\gamma} \Psi)(q;x)=\int_{\G}\!\Phi\(\tau(y)^{-1}\tau(x)q;y\)
\Psi\big(\tau(y^{-1}x)^{-1}y^{-1}\tau(x)q;y^{-1}x\big)\gamma\big(\tau(x)q;y,y^{-1}x\big)\,d\m(y)\,,
\end{equation*}
\begin{equation*}\label{villibrant}
\Phi^{\diamond^{\tau}_{\gamma}}(q;x)=\overline{\gamma\big(\tau(x)q;x,x^{-1}\big)}\;\overline{\Phi(\tau(x^{-1})^{-1} x^{-1}\tau(x)q;x^{-1})}\,.
\end{equation*}
The space $L^{1}(\G,\A) $ is a Banach ${}^*$-algebra; its enveloping $C^*$-algebra $\A\!\rtimes_{{\sf L},\gamma}^{\,\tau}\!\G$ will be denoted by $\mathfrak C^{\tau}\!(\A,\gamma)$\,.

\medskip
In our case $\A$ is a left-invariant $C^*$-algebra of functions on $\G$ and this has remarkable consequences.

\begin{Proposition}\label{zischen}
Let $\A$ be a left invariant $C^*$-algebra of left uniformly continuous functions and $\,\U(\A)$ the Abelian Polish $\G$-module formed of all the unitary elements of the multiplier $C^*$-algebra of $\A$\,. 
\begin{enumerate}
\item[(i)]
Any $2$-cocyle $\,\gamma:\G\times\G\to\U(\A)$ is pseudo-trivializable.
\item[(ii)]
If $\,\beta_1,\beta_2:\G\to C(\G;\T)$ are two pseudo-trivializations of $\,\gamma$\,, there exists a unique $a\in C(\G;\T)$ such that $\beta_2=\delta^0(a)\beta_1$\,.
\end{enumerate}
\end{Proposition}

\begin{proof}
(i) It is a consequence of the case $n=2$ of Proposition \ref{cohomology}, using the larger $\G$-module $\,{\cal U}'=C(\G;\mathbb{T})$\,, also with the left translation action of $\,\G$\,. Note that a pseudo-trivialization $\beta$ can be defined by
\begin{equation}\label{given}
\beta_\gamma:\G\times\G\to\mathbb T\,,\quad\beta_\gamma(q;x):=\gamma\big(\e;q^{-1}\!,x\big)\,.
\end{equation}

\medskip
(ii) Working with $\,{\cal U}'=C(\G;\mathbb{T})$\,, one has $\,\delta^1(\beta_1)=\delta^1(\beta_2)$ if and only if $\,\delta^1\big(\beta_2\beta_1^{-1}\big)=1$ (pointwise operations). Thus it is enough to show that $\delta^1(\beta)=1$ if and only if $\beta=\delta^0(a)$ for some $a\in C(\G;\T)$\,.  

The "if\," part is a direct verification relying on the definitions of $\delta^0$ and $\delta^1$. The "only if\," part is a direct consequence of Proposition \ref{cohomology}: use \eqref{miss} for the case $n=1$ to obtain 
$$
a(q):=\beta\big(\e;q^{-1}\big)\,,\quad q\in\G\,.
$$
\end{proof}

Thus for a  compatible basic data $(\A,\gamma)$ the cocycle is pseudo-trivial.  However, in most cases $\beta_\gamma$ is not $\U(\A)$-valued, and in certain situations it is essential to keep track of the $\A$-behavior. 

\medskip
Let $\beta$ be a pseudo-trivialization of the $2$-cocycle $\gamma$\,, i.e. $\beta$ is a continuous map from $\G\times\G$ to $\T$ satisfying
\begin{equation}\label{vrumushel}
\gamma(q;y,z)=\beta(q;y)\beta(y^{-1}q;z)\beta(q;yz)^{-1}\,,\quad\forall\,q,y,z\in\G\,.
\end{equation} 
Then we have a natural covariant representation $\big(\rho,T_\beta,\H\big)$\,, called {\it the Schr\"odinger representation} defined by $\beta$, given in $\mathcal H:=L^2(\G)$ by
\begin{equation*}\label{razvan}
\[T_\beta(y)u\]\!(q):=\beta(q;y)u\!\(y^{-1}q\)\,,\ \quad \rho(a)u:=a u\,.
\end{equation*}
Applying Definition \ref{peter}, the integrated form $\Sch^\tau_\beta:=\rho\rtimes^{\tau} T_\beta$ is given for $\Phi\in L^1(\G;\A)$\,, $u\in L^2(\G)$ by 
\begin{equation}\label{rada}
\begin{aligned}
\[\Sch^\tau_\beta(\Phi)u\]\!(q)=&\int_\G\beta(q;z)\Phi\big(\tau(z)^{-1}q;z\big)u(z^{-1}q)\,d\m(z)\\
=&\int_\G\beta(q;qy^{-1})\Phi\!\left(\tau(qy^{-1})^{-1}q;qy^{-1}\right)\!u(y)\,d\m(y)\,. 
\end{aligned}
\end{equation}
It is convenient to rewrite the previous formula as an integral operator. To do this we define 
\begin{equation*}\label{ruj}
[{\sf Int}(M)u](x):=\int_\G M(x,y)u(y)d\m(y)
\end{equation*}
and the maps (changes of variables) $\,\c,\,\v^{\tau}:\G\times\G\to\G\times\G$ given by
\begin{equation}\label{genge}
\c(q,x):=\big(q,qx^{-1}\big)\,,\quad\v^{\tau}(q,x):=\big(\tau(x)^{-1}q,x\big)\,.
\end{equation}
With this one can rewrite 
\begin{equation}\label{zurprise}
\Sch^\tau_\beta\!={\sf Int}\circ\C \circ M_\beta\circ \V^{\tau},
\end{equation}
where $M_\beta$ is the operation of multiplication with $\beta$\,, considering $\beta$ as a map acting in two variables. The  operators $\C,\,\V^{\tau}$ are compositions with the maps $\c$ and $\v^{\tau}$ respectively. For different $\tau$'s one has
\begin{equation}\label{schrtau}
\Sch^{\tau'}_\beta\!=\Sch^\tau_\beta\circ(\V^{\tau})^{-1}\!\circ\V^{\tau'}\,.
\end{equation}

\begin{Remark}
{\rm If $\tau(\cdot)=\e$ the operator $\V^{\tau}$ is the identity. Also, if $\beta\equiv 1$\,, (\ref{zurprise}) can be rewritten as $\Sch^\tau_\beta\!={\sf Int}\circ\C \circ \V^{\tau}$. The map $\C\circ\V^{\tau}$ can be seen in $L^{2}(\G\times\G)$ as the composition  with the change of variable
$(q,x)\mapsto(\tau(x)^{-1}q,qx^{-1})$\,. Thus $\C\circ\V^{\tau}$ can be identified with the operator ${\rm CV}^{\tau}$ of \cite{MR}.
}
\end{Remark}

The following proposition can be seen as a sort of covariance under the pseudo-trivialization choice.

\begin{Proposition}\label{tarnafas}
Let $(\A,\gamma)$ be a compatible basic data and $\beta,\beta':\G\to C(\G;\T)$ two pseudo-trivializations of $\,\gamma$\,, connected as in Proposition \ref{zischen} by $\,\beta'=\delta^0(a)\beta$\,. Define the unitary operator $M_a:L^2(\G)\to L^2(\G)$ of multiplication by $a$\,. 
For every $\Phi\in \mathfrak C^\tau\!(\A,\gamma)$ one has
\begin{equation}\label{firsta}
{\sf Sch}^{\tau}_{\beta'}(\Phi)=M_a^*\,\Sch^\tau_\beta(\Phi)\,M_a\,.
\end{equation}
\end{Proposition}

\begin{proof}
Since $\Sch^\tau_\beta=\rho\rtimes^\tau\!T_\beta$\,, we first compute for $u\in L^{2}(\G)$ and $y\in\G$
\begin{eqnarray*}
\big[T_{\beta'}(y)u\big](q) &=& [\delta^0(a)\beta](q;y)u\!\(y^{-1}q\)  \\
&=& a(q)^{-1}\beta(q;y)a(y^{-1}q)u\!\(y^{-1}q\) \\
&=& \big[M_{a}^*\,T_\beta(y) M_a u\big](q)\,.
\end{eqnarray*}
Clearly, $M_{a}$ commutes with $\rho(b)$ for every $b\in\A$\,. Then, by the explicit form of $\Sch^\tau_\beta(\Phi)$\,, (\ref{firsta}) holds for $\Phi\in L^{1}(\G,\A)$\,. By the universal property of the enveloping procedure, it also holds for $\Phi\in \mathfrak C^\tau\!(\A,\gamma)$\,.
\end{proof}

\section{Twisted pseudodifferential operators}\label{firtanunus}

\subsection{The non-commutative Fourier transform}\label{frefelinc}

To switch to the setting of pseudo-differential operators, we need more assumptions on the group $\G$\,, allowing a manageable Fourier transformation. We refer to  \cite{Di,Fo1} for a systematic presentation of the Harmonic Analysis concepts briefly outlined below. 

We set $\,\wG:={\rm Irrep(\G)}/_{\cong}$ for {\it the unitary dual of $\,\G$}\,, composed of unitary equivalence classes of strongly continuous irreducible Hilbert space representations $\pi:\G\rightarrow\mathbb U(\H_\pi)$\,.  There is a standard Borel structure on $\wG$\,, called {\it the Mackey Borel structure} \cite[18.5]{Di}.  The unitary dual $\wG$ is also a separable locally quasi-compact topological space \cite[18.1]{Di}. If $\G$ is Abelian, $\wG$ is the Pontryagin dual group; if not, $\wG$ is not a group.
We denote by $C^*(\G)$ the full (universal) $C^*$-algebra of $\G$ and by $C^*_{\rm red}(\G)\subset\mathbb B\big[L^2(\G)\big]$ its reduced $C^*$-algebra. Any representation $\pi$ of $\G$ generates canonically a non-degenerate representation $\Pi$ of $C^*(\G)$\,.

\begin{Definition}\label{uja}
The locally compact group $\G$ is {\rm type I} if for every irreducible representation $\pi$ one has $\mathbb K(\H_\pi)\subset\Pi\big[C^*(\G)\big]$\,.
It will be called {\rm admissible} if it is second countable, type I and unimodular.
\end{Definition}

For  the concept of {\it type I group} and for examples we refer to \cite{Di,Fo1,Fu} or \cite[Sect. 2]{MR}.  The main consequence of this property is the existence of a measure on the unitary dual $\wG$ for which a Plancherel Theorem holds. This is {\it the Plancherel measure associated to $\m$}\,, denoted by $\wm$ \cite[18.8]{Di}. 

It is known that there is a $\wm$-measurable field $\big\{\,\H_\xi\mid\xi\in\wG\,\big\}$ of Hilbert spaces  and a measurable section $\wG\ni\xi\mapsto\pi_\xi\in{\rm Irrep(\G)}$ such that each $\pi_\xi:\G\rightarrow\mathbb B(\H_\xi)$ is an irreducible representation belonging to the class $\xi$\,. By a systematic abuse of notation, instead of $\pi_\xi$ we will write $\xi$\,, identifying irreducible representations (corresponding to the measurable choice) with elements of $\wG$\,.  

The Fourier transform \cite[18.2]{Di} of $u\in L^1(\G)$ is defined as
\begin{equation*}\label{ion}
({\fscr F}u)(\xi)\equiv \widehat{u}(\xi):=\int_\G u(x)\xi(x)^*d\m(x) \in\mathbb B(\H_\xi)\,.
\end{equation*} 
It defines an injective linear contraction ${\fscr F}:L^1(\G)\rightarrow \mathscr B(\wG)$\,, where $\mathscr B(\wG):=\int^\oplus_{\wG}\mathbb B(\H_\xi)d\wm(\xi)$ is a direct integral von Neumann algebra. One also introduces the direct integral Hilbert space 
\begin{equation*}\label{andy}
\mathscr B^2(\wG):=\int_{\wG}^\oplus\!\B^2(\H_\xi)\,d\wm(\xi)\,\cong\int_{\wG}^\oplus\!\H_\xi\otimes\overline\H_\xi\,d\wm(\xi)\,,
\end{equation*}
with the scalar product
\begin{equation*}\label{justin}
\<\phi_1,\phi_2\>_{\mathscr B^2(\wG)}:=\int_{\wG}\,\<\phi_1(\xi),\phi_2(\xi)\>_{\mathbb B^2(\H_\xi)}d\wm(\xi)=\int_{\wG}{\rm Tr}_\xi\!\[\phi_1(\xi)\phi_2(\xi)^*\]d\wm(\xi)\,,
\end{equation*}
where ${\rm Tr}_\xi$ is the usual trace in $\mathbb B(\H_\xi)$\,.
A generalized form of Plancherel's Theorem \cite{Di,Fo,Fu} states that {\it the Fourier transform ${\fscr F}$ extends from $L^1(\G)\cap L^2(\G)$ to a unitary isomorphism 
${\fscr F}:L^2(\G)\rightarrow \mathscr B^2(\wG)$}\,. The explicit formula of the inverse of the Fourier transform is given by
\begin{equation}\label{marian}
\big({\fscr F}^{-1}\phi\big)(x)= \int_{\wG} {\rm Tr}_\xi\!\[\xi(x)\phi_1(\xi)\]d\wm(\xi)
\end{equation}
for $\phi\in\mathscr B^2(\wG)\cap\mathscr B^1(\wG)$\,, where $\mathscr B^1(\wG)$ denotes the space of sections $\phi$ with $\int_{\wG}{\rm Tr}_\xi\big[|\phi(\xi)|\big]\,d\wm(\xi)<\infty\,$. 

Below we will also need the notations $\,\Gamma:=\G\times\wG$ and $\,\widehat\Gamma:=\wG\times\G$ (in general they are not dual to each other). We also introduce the Hilbert space tensor products 
\begin{equation}\label{recal}
\mathscr B^2(\Gamma):=L^2(\G)\otimes\mathscr B^2(\wG)\,,\quad\mathscr B^2(\widehat\Gamma):=\mathscr B^2(\wG)\otimes L^2(\G)\,.
\end{equation}
The notations are also justified by the fact that these spaces can be written as direct integrals in an obvious way, over $\Gamma$ or $\widehat\Gamma$ respectively.

\subsection{The twisted $\tau$-quantization}\label{frefelin}

Assuming that $\G$ is an admissible group, we reconsider a compatible data $(\A,\gamma)$ with a pseudo-trivialization $\beta$ and the associated Schr\"odinger representation $\Sch^{\tau}_{\beta}$\,. We already mentioned that $L^1(\G;\A)$ can be identified with the completed projective tensor product $\A\,\overline\otimes\,L^1(\G)$\,. Then, by \cite[Ex.\! 43.2]{Tr}, one gets a linear continuous injection 
\begin{equation*}\label{absenta}
{\sf id}_\A\,\overline\otimes\,{\fscr F}:\A\,\overline\otimes\,L^1(\G)\rightarrow\A\,\overline\otimes\,\mathscr B(\wG)
\end{equation*}
and endows the image space $\big({\sf id}_\A\,\overline\otimes\,{\fscr F}\big)\big[\A\,\overline\otimes\,L^1(\G)\big]$ with the Banach $^*$-algebra structure transported from $L^1(\G;\A)\cong \A\,\overline\otimes\,L^1(\G)\,$ through $\,{\sf id}_\A\,\overline\otimes\,{\fscr F}$.  Explicitly one define on the space $\big({\sf id}_\A\,\overline\otimes\,{\fscr F}\big)\!\[\A\,\overline\otimes\,L^1(\G)\]$ 
\begin{equation}\label{ute}
f\#^\tau_\gamma\,g:=\({\sf id}_\A\,\overline\otimes\,{\fscr F}\)\!\big\{\big(\[{\sf id}_\A\,\overline\otimes\,{\fscr F}\]^{-1}\!f \big)\di^\tau_{\gamma}\!\big(\[{\sf id}_\A\,\overline\otimes\,{\fscr F}\]^{-1}\!g\big)\big\}\,,
\end{equation} 
 as well as the involution 
\begin{equation}\label{nicoleta}
f^{\#^\tau_\gamma}\!:=\[{\sf id}_\A\,\overline\otimes\,{\fscr F}\]\[(\[{\sf id}_\A\,\overline\otimes\,{\fscr F}\]^{-1}f)^{\di^\tau_{\gamma}}\]\,.
\end{equation} 
Then the space $({\sf id}_\A\,\overline\otimes\,{\fscr F})\[\A\,\overline\otimes\,L^1(\G)\]$ is a Banach *-algebra with the norm 
\begin{equation*}\label{sharpnorm}
\|f\|_{\#}:=\|\({\sf id}_\A\,\overline\otimes\,{\fscr F}\) ^{-1}f\,\|_{\mathfrak C^{\tau}\!(\A,\gamma)}\,.
\end{equation*}

We denote by $\mathfrak B^{\tau}\!(\A,\gamma)$ the enveloping $C^*$-algebra of the Banach $^*$-algebra $({\sf id}_\A\,\overline\otimes\,{\fscr F})\big[\A\,\overline\otimes\,L^1(\G)\big]$\,. Recall that we have denoted by $\mathfrak C^{\tau}\!(\A,\gamma)$ the enveloping $C^*$-algebra $\A\!\rtimes_{{\sf L},\gamma}^\tau\!\G$ of the Banach $^*$-algebra $L^1(\G;\A)$\,. By the universal property of the enveloping functor, ${\sf id}_\A\,\overline\otimes\,{\fscr F}$ extends to an isomorphism $\,\mathfrak F_\A:\mathfrak C^\tau\!(\A,\gamma)\rightarrow\mathfrak B^\tau\!(\A,\gamma)$\,. 

\medskip
Composing $\Sch^\tau_\beta$ with the inverse partial Fourier transform, we get a pseudo-differential realisation:
\begin{equation}\label{surprize}
\Op^\tau_\beta:=\Sch^\tau_\beta\!\circ({\sf id}\otimes{\fscr F})^{-1}={\sf Int}\circ\C\circ  M_\beta\circ\V^{\tau}\circ\big({\sf id}\otimes{\fscr F}^{-1}\big)\,.
\end{equation}
By extension $\,\Op^\tau_\beta:=\Sch^\tau_\beta\circ\mathfrak F_\A^{-1}\,$ defines a $^*$-representation of $\mathfrak B^\gamma_{\!\A}$ in the Hilbert space $L^2(\G)$\,. Explicitly
\begin{equation}\label{bilfret}
\[\Op^\tau_\beta(f)u\]\!(q)=\int_\G\!\int_{\wG}\,\beta(q;qy^{-1}){\rm Tr}_\xi\Big[\xi(qy^{-1})f \big(\tau(qy^{-1})^{-1}q,\xi\big)\Big]u(y)d\m(y)d\wm(\xi)\,.
\end{equation}
The formula \eqref{bilfret} is rigorously correct if, for instance, the symbol $f$ belongs to $({\sf id}\otimes{\fscr F})[C_{\rm c}(\G\times\G)]$\,, since the explicit form \eqref{marian} of the inverse Fourier transform holds on ${\fscr F}[C_{\rm c}(\G)]\subset\mathscr B^1(\wG)\cap\mathscr B^2(\wG)$\,. Such type of operators appear in \cite{MR} only for the untwisted case $\beta(\cdot;\cdot)=1$\,.

The following consequence of the Proposition \ref{tarnafas} describes the dependence on the pseudo-trivialization.

\begin{Proposition}\label{opbeta}
Let $(\A,\gamma)$ be a compatible basic data and $\beta,\beta':\G\to C(\G;\T)$ two pseudo-trivializations of $\gamma$\, such that $\,\beta'=\delta^0(a)\beta\,$ (cf. Prop. \ref{zischen} (ii))\,. If the group $\G$ is admissible, for each $f\in\mathfrak B^\gamma_\A$ one has
\begin{equation*}\label{seca}
\Op^\tau_{\beta'}(f)=M_a^*\,\Op^\tau_\beta(f)\,M_a\,,
\end{equation*}
where $M_{a}$ is the operator multiplication by $a$ (as in Proposition  \ref{tarnafas}).
\end{Proposition}

\begin{Remark}\label{dalgebra}
{\rm By Remark \ref{tauiso} and the enveloping procedure, one gets isomorphisms $\mu_{\beta}^{\tau,\tau'}: \mathfrak B^{\tau}\!(\A,\gamma)\to\mathfrak B^{\tau'}\!(\A,\gamma)$ for different maps $\tau$ and $\tau'$, which satisfy $\Op^\tau_{\beta}=\Op^{\tau}_{\beta}\circ\mu_{\beta}^{\tau,\tau'}$. }
\end{Remark}

\begin{Remark}\label{resharp}
{\rm The involution and multiplication in the algebra $\mathfrak B^{\tau}\!(\A,\gamma)$ are defined to satisfy the relations
\begin{equation*}\label{kau}
\Op^\tau_\beta(f\#^{\tau}_{\gamma}\,g)=\Op^\tau_\beta(f)\Op^\tau_\beta (g)\, \quad {\rm and} \quad \Op^\tau_\beta (f^{\#^{\tau}_{\gamma}}) = \Op^\tau_\beta(f)^{*}.
\end{equation*}
}
\end{Remark}

\subsection{Twisted convolution operators}\label{twistconop}

Consider a compatible basic data $(\A,\gamma)$\,, its $\tau$-Sch\"odinger representation  $\Sch^{\tau}_{\beta}$ associated to a pseudo-trivialization $\beta$ of $\gamma$ and its extension to the multiplier algebra $\M\big[\mathfrak C^{\tau}\!(\A,\gamma)\big]$\,. 

For $a\in\M(\A)$ and $w\in L^{1}(\G)$ we set $a\otimes w\in L^{1}[\G,\M(\A)]$\,, where 
$$
(a\otimes w)(q;x):=a(q)w(x)\,.
$$ 
If we denote by $1_{\A}$ the unit in $\A$\,, i.e. $1_{\A}(q)=1$ for all $q\in\G$\,, one can check  \cite{BM} that $\M\big[\mathfrak C^{\tau}\!(\A,\gamma)\big]$ contains elements of the form $1_{\A}\otimes w$ for all $w\in L^{1}(\G)$\,, thus one can embed $L^{1}(\G)$ in $\M\big[\mathfrak C^{\tau}\!(\A,\gamma)\big]$\,. 

Setting $\Sch_{\beta}=\Sch^{\tau}_{\beta}$ if $\tau\equiv{\sf e}$\,,  the {\it twisted convolution operator} ${\sf Conv}_{\beta}:L^{1}(\G)\to\mathbb{B}[L^{2}(\G)]$ is 
\begin{equation}\label{convdef}
[{\sf Conv}_{\beta}(w)u](q):=[\Sch_{\beta}(1_{\A}\otimes w)u](q)=\int_{\G}\beta(q;z)w(z)u(z^{-1}q)\,d\m(z)\,.
\end{equation}
They are no longer invariant operators. We refer to Remark \ref{intors} for another point of view. For $v,\,w\in L^{1}(\G)$ one has
\begin{equation*}
{\sf Conv}_{\beta}(v){\sf Conv}_{\beta}(w) = \Sch_{\beta}[(1_{\A}\otimes v)] \Sch_\beta[(1_{\A}\otimes w)]= \Sch_{\beta}[(1_{\A}\otimes v)\diamond^{{\sf e}} _{\gamma}(1_{\A}\otimes w)]\,.
\end{equation*}
If $\gamma$ is trivial the choice $\beta=1$ is legitimate and ${\sf Conv}(w)$ is nothing but the left convolution by $w$ and one has ${\sf Conv}(v){\sf Conv}(w)={\sf Conv}(v*w)$ \cite{MR}. For general $\gamma$\,, a straightforward computation leads to
$$
[{\sf Conv}(v){\sf Conv}(w)u](q)=\int_\G\beta(q;x)\Big[\int_\G \gamma\big(q;xy^{-1}\!,y\big)v(xy^{-1})w(y) d\m(y)\Big]u(x^{-1}q)d\m(x)
$$
and this no longer a twisted convolution operator, since the inner integral also depends on $q$\,.

\medskip
We now want to compute $\Op^\tau_{\beta}(a\otimes\widehat w)$\,, where $a$ is a bounded uniformly continuous function and $\widehat w$ is the Fourier transform of a function in $L^{1}(\G)$. For general $\tau$ the formula is too complicated. But considering the case of a constant map $\tau(\cdot)\equiv x_0$ for some $x_0\in\G$\,, one obtains
$$
\begin{aligned}
([\Op^{x_0}_{\beta}(a\otimes\widehat w)]u)(q) &= a(x_0^{-1}q)\int_{\G}\beta(q;qy^{-1})(\mathscr F ^{-1}\widehat w)(qy^{-1})u(y)\,
d\m(y) \\
&= a(x_0^{-1}q)\int_{\G}\beta(q;z)w(z)u(z^{-1}q)d\m(z) \,,
\end{aligned}
$$
which can be rewritten as
\begin{equation}\label{opbasice}
\Op^{x_0}_{\beta}(a\otimes\widehat w)= M_{{\sf L}_{x_0}\!(a)}\circ{\rm Conv}_{\beta}(w)
\end{equation}
where $M_{{\sf L}_{x_0}\!(a)}$ is the multiplication operator by the function ${\sf L}(x_0)a$\,. 
If $x_0={\sf e}$\,, in the quantization $\Op_{\beta}\equiv\Op^\e_{\beta}$ the multiplication operators stay at the left and twisted convolutions to the right. 
Analogously, for $\tau={\sf id}$ one gets 
$$
\big(\big[\Op^{{\sf id}}_{\beta}(a\otimes\widehat w)\big]u\big)(q) = \int_{\G}\beta(q;qy^{-1})w(qy^{-1})a(y)u(y)\,d\m(y)\,, 
$$
and thus, with an opposite ordering,
\begin{equation*}
\Op^{{\sf id}}_{\beta}(a\otimes\widehat w)={\sf Conv}_{\beta}(w)\circ M_{a}\,.
\end{equation*}

\subsection{From $\Op^{\tau}_{\beta}$ to a twisted Weyl system}\label{ws}

In this section we will study the quantization $\Op^{\tau}_{\beta}$ for symbols $f\in L^{2}(\G)\otimes \mathscr B^{2}(\wG)$\,, recovering it by an integration procedure from a twisted Weyl system, a generalization of phase-space translations from the case $\G=\R^n$\,. We will keep the notations of the previous sections and fix the algebra $\A=C_{0}(\G)$ for which it is known that $\mathfrak C^{\tau}[C_{0}(\G),\gamma]=C_{0}(\G)\!\rtimes_{{\sf L},\gamma}^{\tau}\!\G$ is isomorphic to the compact operators on $L^{2}(\G)$\,: 
one has for any pseudo-trivialization $\beta$ of the $2$-cocycle $\gamma$
$$ 
L^{2}(\G)\otimes\mathscr B^{2}(\wG)\cong L^{2}(\G)\otimes\overline{L^{2}(\G)}\cong \mathbb{B}^{2}[L^{2}(\G)] \subset 
\mathbb{K}[L^{2}(\G)]=\Op^{\tau}_{\beta}\big[\mathfrak B^{\tau}(C_{0}(\G),\gamma)\big]\,.
$$

We start examining the operator  $\Op^{\tau}_{\beta}(f)$ for $f\in (\id\otimes\mathscr F)[C_{c}(\G\times\G)]$\,, given by
\begin{equation}\label{keropi}
\Op^{\tau}_{\beta}(f)=\big[{\sf Int}\circ {\sf Ker}^{\tau}_{\beta}\big](f)\,,
\end{equation}
where, by \eqref{surprize}, the integral kernel is $\big[{\sf C}\circ M_{\beta}\circ\V^{\tau}\circ(\id\otimes{\fscr F}^{-1})\big](f)$ and given formally by
$$
\big[{\sf Ker}^{\tau}_{\beta}(f)\big](x,y)=\int_{\wG} \beta(x;xy^{-1}){\sf Tr}_{\xi}\!\[ \xi(xy^{-1})f\big(\tau(xy^{-1})^{-1}x,\xi\big)\]d\wm(\xi)\,.
$$
It is clear that the operators $M_{\beta}$, $\C$, $\V^{\tau}$ are unitary in $L^{2}(\G\times\G)$\,. Then one has ${\sf Ker}^{\tau}_{\beta}(f)\in L^{2}(\G\times\G)$ and thus $\Op^{\tau}_{\beta}$ is a Hilbert-Schmidt operator in $L^{2}(\G)$\,. Also applying the non-commutative Plancherel theorem, one obtains the unitary map 
$$
L^{2}(\G)\otimes\mathscr B^{2}(\wG)\ni f\mapsto \Op^{\tau}_{\beta}(f)\in\mathbb{B}^{2}[L^{2}(\G)]\,.
$$

\begin{Definition}\label{verners}
For $u,v\in L^2(\G)$\,, we call $\,\mathcal V^{\tau,\beta}_{u,v}$ the unique element of $L^{2}(\G)\otimes\mathscr B^{2}(\wG)$ wich corresponds via $\Op^{\tau}_{\beta}$ to the rank one operator $\,w\mapsto\Lambda_{u,v}(w):=\<w,u\>v\,$.
\end{Definition}

\begin{Proposition}\label{cuyu}
For all $u,v \in L^{2}(\G)$ one has
$$
\mathcal V^{\tau,\beta}_{u,v}=\[(\id\otimes\mathscr F)\circ ({\sf V}^{\tau})^{-1}\circ M_{\beta^{-1}} \circ {\sf C}^{-1}\](v\otimes\overline u)\,.
$$
\end{Proposition}

\begin{proof}
For $f\in L^{2}(\G)\otimes\mathscr B^{2}(\wG)$ we compute using the fact that $\Op^{\tau}_{\beta}$ is unitary
\begin{eqnarray*}
\big\<f,\mathcal V^{\tau,\beta}_{u,v}\big\>_{L^{2}(\G)\otimes\mathscr B^{2}(\wG)} &=&\big\<\Op^{\tau}_{\beta}(f),\Lambda_{u,v}(w)\big\>_{\mathbb B^2[L^{2}(\G)]}= {\sf Tr}[\Op^{\tau}_{\beta}(f)\Lambda_{u,v}^{*}] ={\sf Tr}\big[\Lambda_{v,\Op^{\tau}_{\beta}(f)u}\big] \\
&=& \big\<\Op^{\tau}_{\beta}(f)u,v\big\>_{L^{2}(\G)} = \big\<\big[\big({\sf Int}\circ{\sf Ker}^\tau_\beta\big)(f)\big]u,v\big\>_{L^{2}(\G)} \\
&=& \int_{\G}\(\int_{\G} \big(\[\C\circ M_{\beta}\circ\V^{\tau}\!\circ(\id\otimes{\fscr F}^{-1})\]\!f\big)\!(x,y)u(y)\,d\m(y)\)\overline{v(x)}d\m(x) \,\\
&=&\int_{\G}\int_{\G} \(\[\C\circ M_{\beta}\circ\V^{\tau}\!\circ(\id\otimes{\fscr F}^{-1})\]\!f\)\!(x,y)(\overline v\otimes u)(x,y)\,d\m(y)d\m(x) \\
&=&\big\<\[\C\circ M_{\beta}\circ\V^{\tau}\!\circ(\id\otimes{\fscr F}^{-1})\]\!f, v\otimes\overline u\,\big\>_{L^{2}(\G)\otimes L^{2}(\G)} \\
&=& \big\<\[(\id\otimes{\fscr F}^{-1})\]\!f,\[(\V^{\tau})^{-1}\!\circ M_{\beta^{-1}}\circ \C^{-1} \] (v\otimes\overline u)\big\>_{L^{2}(\G)\otimes L^{2}(\G)} \\
&=& \big\<f, \[(\id\otimes{\fscr F})\circ (\V^{\tau})^{-1}\!\circ M_{\beta^{-1}}\circ {\sf C}^{-1}\] (v\otimes\overline u)\big\>_{L^{2}(\G)\otimes\mathscr B^{2}(\wG)} \,,
\end{eqnarray*}
finishing the proof.
\end{proof}

One can write formally 
\begin{equation}\label{io}
\mathcal V^{\tau,\beta}_{u,v}(x,\xi)= \int_{\G}\,\overline{\beta(\tau(y)x;y)}\,v(\tau(y)x)\,\overline{u(y^{-1}\tau(y)x)}\,\xi(y)^{*}\,d\m(y)\,.
\end{equation}
This formula hods, for example, if $u,\, v\in C_{c}(\G)$\,.

\begin{Proposition}\label{rubincar}
The transformation $f\mapsto\Op^{\tau}_{\beta}(f)$ maps unitarily the space $L^2(\G)\otimes\mathscr B^{2}(\wG)$ in the  Hilbert-Schmidt class on $L^{2}(\G)$. Moreover, $\Op^{\tau}_{\beta}(f)$ is the unique bounded linear operator in $L^2(\G)$ associated by the relation $\op^\tau_\beta[f](u,v)=\big\<\Op^\tau_\beta(f) u,v\big\>_{L^2(\G)}$
to the bounded sesquilinear form
\begin{equation}\label{fernan}
\op^\tau_\beta[f]:L^2(\G)\times L^2(\G)\rightarrow\mathbb C\,,\quad\op^\tau_\beta[f](u,v):=\big\< f,\mathcal V^{\tau,\beta}_{u,v}\big\>_{\!\mathscr B^2(\Gamma)}\,.
\end{equation}
\end{Proposition}

\begin{proof}
This is just a summary of results already obtained above. The identity 
$$
\big\<\Op^\tau_\beta(f) u,v\big\>_{L^2(\G)}=\big\< f,\mathcal V^{\tau,\beta}_{u,v}\big\>_{\!\mathscr B^2(\Gamma)}
$$ 
has been obtained during the proof of Proposition \ref{cuyu}.
\end{proof}

We also introduce
\begin{equation}\label{wotro}
\mathcal W^{\tau,\beta}_{u,v}:=(\mathscr F\otimes\mathscr F^{-1})\mathcal V^{\tau,\beta}_{u,v}=\[(\mathscr F\otimes \id)\circ ({\sf V}^{\tau})^{-1}\!\circ M_{\beta^{-1}} \circ {\sf C}^{-1}\](v\otimes\overline u)\,,
\end{equation}
for which, by the Plancherel theorem and using the notation $\widehat f:=(\mathscr F\otimes\mathscr F^{-1})f$\,, one has the identity
\begin{equation*}\label{wdef}
\<\Op^{\tau}_{\beta}(f)u,v\>_{L^{2}(\G)} =\big\<\widehat f,\mathcal W^{\tau,\beta}_{u,v}\big\>_{\mathscr B^{2}(\wG)\otimes L^{2}(\G)}\,,\quad u,v\in L^2(\G)\,.
\end{equation*}
We note the {\it orthogonality relations} (recall \eqref{recal})
\begin{equation*}\label{orthogo}
\big\<\mathcal W^{\tau,\beta}_{u,v},\mathcal W^{\tau,\beta}_{u',v'}\big\>_{\mathscr B^{2}(\widehat\Gamma)}=\<u',u\>_{L^{2}(\G)}\<v,v'\>_{L^{2}(\G)}={\sf Tr}[\Lambda_{u,v}\Lambda_{v',u'}]=\big\<\mathcal V^{\tau,\beta}_{u,v},\mathcal V^{\tau,\beta}_{u'\!,v'}\big\>_{\mathscr B^{2}(\Gamma)}\,.
\end{equation*}

\begin{Definition}\label{graak}
For each $x\in \G$ and $\xi\in\wG$ we define ${\sf W}^{\tau}_{\beta}(\xi,x)$ to be the unique bounded linear operator in $L^{2}(\G,\H_{\xi}):=L^{2}(\G)\otimes\H_{\xi}$ satisfying for all $u,v\in L^2(\G)$ and $\varphi_{\xi},\psi_{\xi}\in\H_\xi$
\begin{equation*}\label{draak}
\big\<\mathcal W^{\tau,\beta}_{u.v}(\xi,x)\varphi_{\xi},\psi_{\xi}\big\>_{\H_{\xi}}=\big\<{\sf W}^{\tau}_{\beta}(\xi,x)(\overline u\otimes\varphi_{\xi}),\overline v\otimes\psi_{\xi}\big\>_{L^{2}(\G,\H_{\xi})}.
\end{equation*}
The family $\big\{{\sf W}^{\tau}_{\beta}(\xi,x)\big\}_{\!(\xi,x)\in\wG\times\G}\,$ will be called {\rm \,the $\tau$-twisted Weyl system associated to the pseudo-trivialization $\beta$\,}.
\end{Definition}

\begin{Remark}\label{retauriels}
{\rm For an operator $T$ in $L^2(\G;\H_\xi)\cong L^2(\G)\otimes\H_\xi$ and a pair of vectors $u,v\in L^2(\G)$\,, the action of $\,\left<Tu,v\right>_{L^2(\G)}\in\mathbb B(\H_\xi)$ on $\,\varphi_\xi\in\H_\xi$ is given by 
\begin{equation*}\label{tauriel}
\left<Tu,v\right>_{L^2(\G)}\varphi_\xi:=\int_{\G}[T(u\otimes\varphi_\xi)](y)\overline{v(y)}\,d\m(y)\in\H_\xi\,.
\end{equation*}
With this interpretation one has the  the equality
\begin{equation*}\label{leos}
\W_{u,v}^{\tau,\beta}(\xi,x)=\left<{\sf W}^\tau_\beta(\xi,x)\overline u,\overline v\right>_{L^2(\G)}\in\mathbb B(\H_\xi)\,,
\end{equation*}
so $(u,v)\to\W_{u,v}^{\tau,\beta}$ can be seen as a {\it "Fourier-Wigner transform"} \cite{Fo} associated to  ${\sf W}^\tau_\beta$\,. Then one can interpret $(u,v)\to\mathcal V_{u.v}^{\tau,\beta}$ as a {\it Wigner transform}. }
\end{Remark}

Using Remark \ref{retauriels} and the equation (\ref{wotro}) one can give the explicit formula for the twisted Weyl system.

\begin{Proposition}\label{bojan}
For $\,\Theta\in L^{2}(\G,\H_{\xi})$ one has
\begin{equation}\label{weylex}
\[{\sf W}^{\tau}_{\beta}(\xi,x)\Theta\]\!(y)=\overline{\beta(y;x)}\xi(y)^*\xi(\tau(x))\[\Theta(x^{-1}y)\]\,.
\end{equation} 
\end{Proposition}

\begin{proof}
One computes for $u,v\in C_{c}(\G)$ and $\varphi_{\xi},\psi_{\xi}\in\H_{\xi}$
$$
\begin{aligned}
\big\<{\sf W}^{\tau}_{\beta}(\xi,x)(u\otimes\varphi_{\xi}), v\otimes\psi_{\xi}\big\>_{L^{2}(\G,\H_{\xi})} &= \big\<\mathcal W^{\tau,\beta}_{\overline u,\overline v}(\xi,x)\varphi_{\xi},\psi_{\xi}\big\>_{\H_{\xi}} \\
&=\big \<\[(\mathscr F\otimes \id)\circ ({\sf V}^{\tau})^{-1}\!\circ M_{\beta^{-1}}\!\circ {\sf C}^{-1}\](\overline v\otimes u)(\xi,x)\varphi_{\xi},\psi_{\xi}\big\>_{\H_{\xi}} \\
&= \int_{\G}\big\<\[({\sf V}^{\tau})^{-1}\!\circ M_{\beta^{-1}}\!\circ {\sf C}^{-1}\]\!(\overline v\otimes u)(z,x)\xi(z)^*\varphi_{\xi},\psi_{\xi}\big\>_{\H_{\xi}}\!d\m(z) \\
&= \int_{\G}\big\<\,\overline{\beta(\tau(x)z;x)}\overline{v(\tau(x)z)}u(x^{-1}\tau(x)z)\xi(z)^*\varphi_{\xi},\psi_{\xi}\big\>_{\H_{\xi}}\!d\m(z) \\
&= \int_{\G}\big\<\,\overline{\beta(y;x)}u(x^{-1}y)\xi^{*}(\tau(x)^{-1}y)\varphi_{\xi},v(y)\psi_{\xi}\big\>_{\H_{\xi}}d\m(z) \\
&=\big\<\,\overline{\beta(y;x)}\xi(\tau(x)^{-1}y)^{*}\!\[(u\otimes \varphi_{\xi})(x^{-1}y)\],v\otimes\psi_{\xi}\big\>_{L^{2}(\G,\H_{\xi})}\,.
\end{aligned}
$$
Thus one has for decomposable vectors
$$
\[{\sf W}^{\tau}_{\beta}(\xi,x)(u\otimes \varphi_{\xi})\]\!(y)=\overline{\beta(y;x)}\,\xi(\tau(x)^{-1}y)^{*}\!\[(u\otimes \varphi_{\xi})(x^{-1}y)\].
$$ 
Then the proof is easily finished by density.
\end{proof}

Recall that $\1\in\wG$ denotes the trivial one dimensional representation. Thus $\H_{\1}\cong\mathbb{C}$ and $L^{2}(\G,\H_{\1})\cong L^{2}(\G)$\,. We assume for simplicity (and because there are no interesting counter-examples) that $\tau(\e)=\e$\,.

\begin{Definition}\label{uvdef}
We define the unitary operators
\begin{equation*}\label{noslipt}
{\sf U}_{\beta}(x):= {\sf W}^{\tau}_{\beta}(x;\1) \in \B[L^{2}(\G)]\,,\quad {\sf V}(\xi):={\sf W}^{\tau}_{\beta}({\sf e},\xi) \in\B[L^{2}(\G,\H_\xi)]\,.\\
\end{equation*}
Explicitly, for $u\in L^{2}(\G)$ and $\Theta\in L^{2}(\G,\H_{\xi})$ one has
$$
[{\sf U}_{\beta}(x)u](y)=\overline{\beta(y,x)}u(x^{-1}y) \;,\quad [{\sf V}(\xi)\Theta](y)=\xi(y)^{*}\Theta(y)\,.
$$
\end{Definition}

We get immediately that 
\begin{equation*}\label{weyluv}
{\sf W}^{\tau}_{\beta}(\xi,x)={\sf V}(\xi)\big({\sf U}_{\beta}(x)\otimes\xi[\tau(x)]\big) \,.
\end{equation*}
It is also easy to prove the commutation relations ($\,\overline{\gamma(x,y)}$ is a multiplication operator in $L^2(\G)$)
\begin{eqnarray*}\label{commutarelone}
 {\sf U}_{\beta}(x){\sf U}_{\beta}(y)\!\!\!&=&\!\!\! \overline{\gamma(x,y)}{\sf U}_{\beta}(xy)\,,\\
\big({\sf V}(\xi)\otimes\id_{\eta}\big)\big(\id_{\xi}\otimes {\sf V}(\eta)\big)\!\!\! &=&\!\!\! \big(\id_{\xi}\otimes {\sf V}(\eta)\big)\big({\sf V}(\xi)\otimes\id_{\eta}\big) \,,\\
\big({\sf U}_{\beta}(x)\otimes\id_{\H_{\xi}}\big) {\sf V}(\xi)\!\!\!&=&\!\!\!{\sf V}(\xi)\big({\sf U}_{\beta}(x)\otimes\id_{\H_{\xi}}\big)\big(\id\otimes\xi(x)\big) \,.
\end{eqnarray*}
The map $x\mapsto {\sf U}_{\beta}(x)$ is a strongly continuous projective representation of the group $\G$ on $L^{2}(\G)$ with a vector valued 2-cocycle $\overline{\gamma}$\,. If $\G$ is Abelian the irreducible representations are all one-dimensional, ${\sf V}$ is a unitary representation of the dual group $\wG$   and the tensor products are no longer needed.

\begin{Remark}\label{intors}
{\rm The integrated form $w({\sf U}_{\beta}) := \int_{\G} w(x) {\sf U}_{\beta}(x)\,d\m(x)$ defined for $w\in L^1(\G)$ leads to
\begin{equation}\label{inteu}
[w({\sf U}_{\beta})u](q)=\int_{\G}\overline{\beta(q,x)}w(x)u(x^{-1}q)\,d\m(x) = \big[{\sf Conv}_{\overline{\beta}}(w)u\big](q)\,,
\end{equation}
so the twisted convolution operators of subsection \ref{twistconop} are recovered in this way. }
\end{Remark}

\subsection{Involutive algebras of symbols}

Since the pseudo-differential calculus for symbols in $L^{2}(\G)\otimes\mathscr B^{2}(\wG)$ is one-to-one, we can define a composition law $\,\#^{\tau}_\gamma$ and an involution $\,^{\#^{\tau}_\gamma}$ on $\mathscr B^2(\Gamma)$ by
\begin{equation*}\label{capsuni}
\Op^{\tau}_{\beta}(f \#^{\tau}_{\gamma}\,g) := \Op^{\tau}_{\beta}(f)\Op^{\tau}_{\beta}(g)\,,\quad\Op^{\tau}_{\beta}(f^{\#^{\tau}_{\gamma}}) := \Op^{\tau}_{\beta}(f)^* .
\end{equation*}
Of course, if $C_{0}(\G)\subset\A$\,, for symbols lying in $(\id\otimes\mathscr F)[\A\otimes L^{1}(\G)]\cap\big[L^{2}(\G)\otimes\mathscr B^{2}(\wG)\big]$ this algebraic structure coincides with the one defined in subsection \ref{frefelin} in the algebra $\mathfrak B^{\tau}\!(\A,\gamma)$\,; in particular, it does not depend on the choice of the pseudo-trivialization $\beta$\,. Using the equation (\ref{keropi}), the composition law can be written in terms of integral kernels as
$\,{\sf Ker}^\tau_\beta({f\#^\tau_\gamma g})={\sf Ker}^\tau_\beta(f)\bullet{\sf Ker}^\tau_\beta(g)$\,, where
\begin{equation*}\label{zuruz}
{\sf Ker}^\tau_\beta\!:=\C \circ M_\beta\circ \V^{\tau}\circ({\sf id}\otimes{\fscr F}^{-1})
\end{equation*} 
and $\,\bullet\,$ is the usual composition of kernels
\begin{equation*}\label{intco}
(M\bullet N)(x,y):=\int_\G M(x,z)N(z,y) d\m(z)\,,
\end{equation*} 
corresponding to ${\sf Int}(M\bullet N)={\sf Int}(M){\sf Int}(N)$\,. It follows that for $f,g\in \mathscr B^{2}(\Gamma)$
\begin{equation*}\label{compoz}
\begin{aligned}
f\#^\tau_\gamma\,g&=\big({\sf Ker}^\tau_\beta\big)^{-1}\big({\sf Ker}^\tau_\beta(f)\bullet{\sf Ker}^\tau_\beta(g)\big)\\
&=({\sf id}\otimes{\fscr F}) ({\sf CMV}^{\tau}_{\beta})^{-1}
\Big\{\big[{\sf CMV}^{\tau}_{\beta}\circ({\sf id}\otimes{\fscr F}^{-1})\big]f \bullet 
\big[{\sf CMV}^{\tau}_{\beta}\circ({\sf id}\otimes{\fscr F}^{-1})\big]g\Big\}\,,
\end{aligned}
\end{equation*}
where ${\sf CMV}^{\tau}_{\beta}:=\C\circ M_{\beta}\circ\V^{\tau}$. Similarly, in terms of the natural kernel involution $M^\bullet(x,y):=\overline{M(y,x)}$ (corresponding to ${\sf Int}(M)^*={\sf Int}(M^\bullet)$)\,, one gets
\begin{equation*}\label{invo}
f^{\#^{\tau}_\gamma}=\big({\sf Ker}^\tau_\beta\big)^{-1}\big[\big({\sf Ker}^\tau_\beta(f)\big)^\bullet\big]=({\sf id}\otimes{\fscr F})\circ({\sf CMV}^{\tau}_{\beta})^{-1}\Big\{\Big(\big[{\sf CMV}^{\tau}_{\beta}\circ({\sf id}\otimes{\fscr F}^{-1})\big]f\Big)^{\!\bullet}\Big\}\,.
\end{equation*}

Let us give the simple algebraic rules satisfied by the twisted Wigner transforms of Definition \ref{verners}\,:
For every $u_1,u_2,v_1,v_2\in L^2(\G)$ one can compute
$$
\begin{aligned}
\Op^{\tau}_{\beta}(\mathcal V^{\tau,\beta}_{u_1,v_1}\#^{\tau}_\gamma\,\mathcal V^{\tau,\beta}_{u_2,v_2}) &= 
\Op^{\tau}_{\beta}(\mathcal V^{\tau,\beta}_{u_1,v_1})\Op^{\tau}_{\beta}(\mathcal V^{\tau,\beta}_{u_2,v_2})= \Lambda_{u_1,v_1}\Lambda_{u_2,v_{2}}\\ 
&= \<v_{2},u_{1}\>\Lambda_{u_{2},v_{1}} = \<v_{2},u_{1}\>\Op^{\tau}_{\beta}(\mathcal V^{\tau,\beta}_{u_2,v_1}) \,.
\end{aligned}
$$
We summarize this and the simpler computation for the adjoint as
\begin{equation*}\label{vizi}
\mathcal V^{\tau,\beta}_{u_1,v_1}\#^\tau_\gamma\,\mathcal V^{\tau,\beta}_{u_2,v_2}=\<v_2,u_1\>\mathcal V^\beta_{u_2,v_1}\quad{\rm and}\quad\big(\mathcal V^{\tau,\beta}_{u,v}\big)^{\#^\tau_\gamma}=\mathcal V^{\tau,\beta}_{v,u}\,.
\end{equation*}

We recall that an involutive algebra $\big(\mathscr B,\#,^{\#}\big)$ is said to be {\it a Hilbert algebra} if it possesses a scalar product $\<\cdot,\cdot\>:\mathscr B\times\mathscr B\to\mathbb{C}$ such that  for every  $g$ the map $g\mapsto g\#f$ is continuous, $\mathscr B\#\mathscr B$ is total in $\mathscr B$ and
$$
\<g^{\#},f^{\#}\>=\<f,g\>\,,\quad\<g\#h,f\>=\<h,g^{\#}\#f\>\,,\quad \forall\,f,g,h\in\mathscr B\,.
$$
A complete Hilbert algebra is called {\it an $H^*$-algebra}. Recall that $\mathscr B^{2}(\Gamma)\!:=\!L^2(\G)\otimes\mathscr B^{2}(\wG)$ was introduced in \eqref{recal}.

\begin{Proposition}\label{gaura}
The space $\Big(\mathscr B^{2}(\Gamma),\#^{\tau}_\gamma,\,^{\#^{\tau}_\gamma},\<\cdot,\cdot\>_{\mathscr B^{2}(\Gamma)}\Big)$ is an $H^{*}$-algebra.
\end{Proposition}

\begin{proof}
One can check this directly. In a simpler way, one can recall the well-known fact  $\mathbb{B}^{2}\!\[L^{2}(\G)\]$ is an $H^{*}$-algebra with the operator composition, the adjoint and the scalar product $\<S,T\>_{\mathbb{B}^{2}[L^{2}(\G)]}={\sf Tr}[ST^{*}]$\,. Then one invokes the unitary $^*$-algebra isomorphism $\Op^{\tau}_{\beta}:L^{2}(\G)\otimes\mathscr B^{2}(\wG)\to\mathbb{B}^{2}\!\[L^{2}(\G)\]$ justified above.
\end{proof}

We now are interesting in the existence of a parameter $\tau$ and a cocycle $\gamma$ leading to a symmetric  quantization. This means that $f^{\#^{\tau}_{\gamma}}=f^{\star}$, where $^{\star}$ is the pointwise involution $f^\star(x,\xi):=f(x,\xi)^*\in \mathbb B(\H_\xi)$\,, or equivalently  $\Op^{\tau}_{\beta}(f^{\star})=\Op^{\tau}_{\beta}(f)^{*}$ (independently of the chosen pseudo-trivialization $\beta$)\,. This permits to show that  ``real valued symbols'' are sent to self-adjoint operators. The case $\gamma\equiv 1$ was discussed in \cite{MR}, where it was shown that a necessary and sufficient condition for symmetric quantization is the existence of a map $\tau$ which satisfies  the relation $\tau(x)=x\tau(x^{-1})$ for all $x\in\G$. If $\tau$ satisfies this relation is called {\it symmetric}. The existence of such a map in general seems to be rather obscure, but some examples are shown in \cite{MR}. 

If there exists a symmetric quantization with non-trivial cocycle we will say that $\G$ possesses a {\it twisted symmetric quantization}. The existence of this kind of quantizations  is unclear, but assuming the existence of a symmetric map $\tau$ one gets the following result.

\begin{Proposition}\label{symmetric}
Suppose that $\G$ admits a symmetric map $\tau$\,. A necessary and sufficient condition for the existence of a symmetric quantization is $\gamma(z,z^{-1})=1_\A\,$ for all $z\in\G$\,.
\end{Proposition}

\begin{proof}
The equality $f^{\star}=f^{\#^{\tau}_{\gamma}}$ can be examined at the level of the respective kernels 
$$
\begin{aligned}
\big[{\sf Ker}^{\tau}_{\beta} (f^{\star})\big](x,y)&= \beta(x;xy^{-1})\big([\id\otimes\mathscr F^{-1}]f^\star\big)(\tau(xy^{-1})^{-1}x;xy^{-1})\\
&=\beta(x;xy^{-1})\overline{\big([\id\otimes\mathscr F^{-1}]f\big)(\tau(xy^{-1})^{-1}x;yx^{-1})}
\end{aligned}
$$
and
$$
\big[{\sf Ker}^{\tau}_{\beta}\big(f^{\#^{\tau}_{\gamma}}\big)\big](x,y)=\big[{\sf Ker}^{\tau}_{\beta}(f)^{\bullet}\big](x,y)= \beta(y;yx^{-1})^{-1}\overline{\big([\id\otimes\mathscr F^{-1}]f\big)(\tau(yx^{-1})^{-1}y;yx^{-1})}\,.
$$
Since the map $\tau$ satisfies the equality $\tau(z)=z\tau(z^{-1})$ for all $z\in\G$\,, the previous expressions coincide for every $f\in L^{2}(\G\times\G)$ if and only if $\beta$ satisfies
\begin{equation}\label{symmbeta}
\beta(x;xy^{-1})\beta(y;yx^{-1})=1\,, \quad \forall\, x,\,y\in\G\,.
\end{equation}
By the pseudo-trivialization choice $\,\gamma=\delta^1(\beta)\,$, and since $\beta(\cdot;\e)=1$\,, one has
$$
\begin{aligned}
\gamma\big(x;xy^{-1}\!,yx^{-1}\big)&=\beta\big(x;xy^{-1}\big)\beta\big(\big[xy^{-1}\big]^{-1}\!x;yx^{-1}\big)\beta\big(x;\big[xy^{-1}\big]\big[yx^{-1}\big]\big)\\
&=\beta(x;xy^{-1})\beta(y;yx^{-1})
\end{aligned}
$$
and everything is clear.
\end{proof}

In the next section we will see examples of $2$-cocycles on nilpotent Lie groups allowing twisted symmetric quantizations.

\section{The case of nilpotent Lie groups}

We suppose now that $\G$ is a nilpotent Lie group, also assumed connected and simply connected. For theory of nilpotent Lie  group we refer to \cite{CG}. These kind of groups are second countable unimodular and of type I, so all  the previous constructions and results apply.   In this section we are going to show that, besides the  ``operator-valued twisted pseudo-differential calculus'' $\Op^{\tau}_{\beta}$ for symbols defined in $\G\times\wG$\,, there is also a ``scalar-valued  twisted pseudo-differential calculus'' $\mathbf{Op}^{\tau}_{\beta}$ which provides a quantization of the cotangent bundle of $\G$\,. We will also put into evidence $2$-cocycles defined by variable magnetic fields.

\subsection{Quantization of scalar symbols on  nilpotent Lie groups}\label{craur}

Let $\mathfrak g$ be the Lie algebra of $\G$ and $\mathfrak g^*$ its dual. If $X\in\g$ and $\mathcal X\in\g^*$ we set $\<X\!\mid\!\mathcal X\>:=\mathcal X(X)$\,. We also denote by $\exp:\mathfrak g\to\G$ the exponential map, which is a diffeomorphism. Its inverse is denoted by $\log:\G\rightarrow\mathfrak g$\,.
Under these diffeomorphisms the Haar measure on $\G$ corresponds to the Haar measure $dX$ on $\g$ (normalized accordingly). It then follows that $L^p(\G)$ is isomorphic to $L^p(\g)$\,: one has a surjective isometry
$$
L^{p}(\G)\overset{{\rm Exp}}{\rightarrow}L^{p}(\mathfrak g)\,,\quad {\rm Exp}(u):=u\circ\exp
$$
with inverse 
$$
L^{p}(\mathfrak g)\overset{{\rm Log}}{\rightarrow}L^{p}(\G)\,,\quad {\rm Log}(\mathbf u):=\mathbf u\circ\log\,.
$$
The Schwartz space $\S(\G)$ is just defined by transport of structure through ${\rm Log}$\,, starting from $\S(\g)$\,.

There is a Fourier transformation defined essentially by
\begin{equation*}\label{clata}
(\mathbf Fu)(\mathcal X)\equiv\hat u(\mathcal X):=\int_{\g}e^{-i\<X\mid\mathcal X\>} u(\exp X)\,dX =\int_\G\!\,e^{-i\<\log x\mid\mathcal X\>} u(x)d\m(x)\,,
\end{equation*}
with inverse
\begin{equation*}\label{clatta}
(\mathbf F^{-1}\mathfrak u)(x)\equiv\check{\mathfrak u}(x):=\int_{\g^*}\!e^{i\<\log x\mid\mathcal X\>} \mathfrak u(\mathcal X)\,d\mathcal X\,.
\end{equation*}
It can be seen as a unitary map $\,\mathbf F:L^2(\G)\rightarrow L^2(\g^*)$ or as a linear topological isomorphism $\,\mathbf F:\mathcal S(\G)\rightarrow\mathcal S(\g^*)$\,. We used a good normalization of the Haar measure on $\g^*$.

Returning to pseudo-differential operators, let us consider a compatible data $(\A,\gamma)$ as in subsection \ref{flegarin} and a continuous map $\tau:\G\to\G$\,.  We compose the Schr\"odinger representation (\ref{rada}) asociated to a pseudo-trivialization $\beta$ of $\gamma $ with the inverse of the partial Fourier transform  
$$
\id\otimes \mathbf F:(L^{1}\cap L^{2})(\G,\A)\to \A\otimes L^{2}(\mathfrak g^*)\,,
$$
finding the pseudo-differential representation 
\begin{equation*}\label{formulare}
\mathbf{Op}^{\tau}_{\beta}:=\Sch^{\tau}_{\beta}\circ (\id\otimes\mathbf F^{-1})\,,
\end{equation*}
which can be extended to the enveloping $C^{*}$-algebra. One gets more or less formally  
\begin{equation*}\label{getsy}
[\mathbf{Op}^{\tau}_{\beta}(\mathbf s)u](x)=\int_{\G}\int_{\mathfrak g^*}  \beta(x;xy^{-1})e^{i\<\log(xy^{-1})\mid\mathcal X\>}\mathbf s\big(\tau(xy^{-1})^{-1}x,\mathcal X\big)u(y)\,d\m(y)d\mathcal X\,.
\end{equation*}
Thus $\mathbf{Op}^{\tau}_{\beta}(\mathbf s)$ is an operator with integral kernel $\mathbf{Ker}^{\tau}_{\beta}[\mathbf s]:\G\times \G\to\mathbb{C}$ given by
$$
\mathbf{Ker}^{\tau}_{\beta}[\mathbf s]:=
\big[{\sf C}\circ M_{\beta}\circ {\sf V}^{\tau}\!\circ (\id\otimes \mathbf F^{-1})\big](\mathbf s)\,.
$$
Examining this kernel, obtained by applying to $\mathbf s$ unitary tranformations, one gets the unitary mapping 
$$
\mathbf{Op}^{\tau}_{\beta}:L^{2}(\G\times\mathfrak g^*)\to\mathbb{B}^2[L^{2}(\G)]\,.
$$

\begin{Remark}
{\rm One can also consider the composition law $\natural^{\tau}_{\gamma}$ defined to satisfy $\mathbf{Op}^{\tau}_{\beta}(\mathbf r\natural^{\tau}_{\gamma}\mathbf s)=\mathbf{Op}^{\tau}_{\beta}(\mathbf r)\mathbf{Op}^{\tau}_{\beta}(\mathbf s)$ and the involution defined to satisfies $\mathbf{Op}^{\tau}_{\beta}(\mathbf s^{\natural^{\tau}_{\gamma}})=\mathbf{Op}^{\tau}_{\beta}(\mathbf s)^{*}$. Then $\big(L^{2}(\G\times\mathfrak g^*),\natural^{\tau}_{\gamma},{}^{\natural^{\tau}_{\gamma}}\big)$ is a $^{*}$-algebra which is isomorphic to $\big(\mathscr B^{2}(\G\times \wG),\#^{\tau}_{\gamma},^{\#^{\tau}_{\gamma}}\big)$ and one has the relation
\begin{equation}
\mathbf{Op}^{\tau}_{\beta}= \Op^{\tau}_{\beta}\circ(\id\otimes\mathscr F)\circ(\id\otimes\mathbf F^{-1})\,.
\end{equation}
Actually this is an isomorphism of $H^*$-algebras, in the sense of Proposition \ref{gaura}. The same Hilbert-Schmidt operators in $L^2(\G)$ can be expressed both by operator-valued symbols defined on $\G\times\wG$ and by scalar symbols defined on the cotangent space $T^*(\G)\cong\G\times\mathfrak g^*$. This is also true for other classes of operators, as those connected with twisted crossed products. 
}
\end{Remark}

The quantization $\mathbf{Op}^{\tau}_{\beta}$ is associated to a new type of twisted Weyl system, indexed by the points of $\G\times\mathfrak g^*$, in terms of which one can write 
$$
\mathbf{Op}^{\tau}_{\beta}(\mathbf s)=\int_{\G}\int_{\mathfrak g^*}\widehat{\mathbf s}(\mathcal X,x)\mathbf W^{\tau}_{\beta}(x,\mathcal X)\,d\m(x)d\mathcal X\,,
$$
using the notation $\,\widehat{\mathbf s}:=(\mathbf F\otimes\mathbf F^{-1})\mathbf s$\,. For this one defines the unitary operator $\mathbf W^{\tau}_{\beta}(x,\mathcal X)$  in $L^{2}(\G)$ by
\begin{equation*}\label{vacio}
\big[\mathbf W^{\tau}_{\beta}(x,\mathcal X)u\big](q):=\beta(q;x)\,e^{i\<\log[\tau(x)^{-1}q]\mid\mathcal X\>}u(x^{-1}q)\,.
\end{equation*}

 Assuming for simplicity that $\tau(\e)=\e$\,, it is more enlightening to define the unitary operators $\mathbf U_{\beta}(x):=\mathbf W^{\tau}_{\beta}(x,0)$ and $\mathbf V(\mathcal X):=\mathbf W^{\tau}_{\beta}({\sf e},\mathcal X)$ in $L^{2}(\G)$\,, explicitly given by
\begin{equation*}\label{belea}
\big[\mathbf U_{\beta}(x)u\big](q)=\beta(q;x)u(x^{-1}q) \,,\quad[\mathbf V(\mathcal X)u](x)=e^{i\<\log q\mid\mathcal X\>}u(q) \,.
\end{equation*}
It is easy to show that for $x,y\in\G$ and $\mathcal X,\mathcal Y\in\mathfrak g^*$
\begin{equation*}
\begin{split}
\mathbf V(\mathcal X)\mathbf V(\mathcal Y) &= \mathbf V(\mathcal Y)\mathbf V(\mathcal X)=\mathbf V(\mathcal X+\mathcal Y)\,, \\
\mathbf U_{\beta}(x)\mathbf U_{\beta}(y) &=\gamma(x,y)\mathbf U_{\beta}(xy)\,,\\
\mathbf U_{\beta}(x)\mathbf V(\mathcal X)&=\Delta(x,\mathcal X)\mathbf V(\mathcal X)\mathbf U_{\beta}(x)\,,
\end{split}
\end{equation*}
where $\,\gamma(x,y)$ is the operator of multiplication by $\,[\gamma(x,y)](\cdot)$ in $L^2(\G)$\,, and $\Delta(x,\mathcal X)$ is the operator of multiplication by 
$$
\G\ni q\to[\Delta(x,\mathcal X)](q):=e^{i\<\log(x^{-1}q)-\log q\mid\mathcal X\>}.
$$

\subsection{Magnetic fields and group cocycles}\label{fourtehinn}

For $x,y\in\G$ one sets $[x,y]:\mathbb R\rightarrow\G$ by
\begin{equation*}\label{abb}
[x,y]_s:=\exp[(1-s)\log x+s\log y]=\exp[\log x+s(\log y-\log x)]\,.
\end{equation*}
The function $[x,y]$ is smooth and satisfies 
$$
[y,x]_s=[x,y]_{1-s}\,,\quad[x,y](0)=x\quad{\rm and}\quad [x,y](1)=y\,.
$$ 
In addition, $[\e,y]$ is a $1$-parameter subgroup passing through $y$\,: one has 
$$
[\e,y]_{s+t}=[\e,y]_s[\e,y]_t\,,\quad\forall\,y\in\G\,,\,s,t\in\R\,.
$$
{\it The segment in $\G$ connecting $x$ to} $y$ is $[[x,y]]:=\big\{\,[x,y]_s\mid s\in[0,1]\,\big\}$\,. 

\medskip
Let us also set $\,\Delta:=\{\,(t,s)\in[0,1]^2\mid s\le t\,\}$\,. For $x,y,z$ one defines the function $\<x,y,z\>:\R^2\rightarrow\G$
\begin{equation*}\label{function}
\<x,y,z\>_{t,s}:=\exp\big[\log x+t(\log y-\log x)+s(\log z-\log y)\big]
\end{equation*}
and the set $\,\<\!\<x,y,z\>\!\>:=\<x,y,z\>_\Delta\subset\G$\,. Note that 
\begin{equation*}\label{potop}
\<x,y,z\>_{0,0}=x\,,\quad\<x,y,z\>_{1,0}=y\,,\quad\<x,y,z\>_{1,1}=z\,,
\end{equation*}
\begin{equation*}\label{potrop}
\<x,y,z\>_{t,0}=[x,y]_t\,,\quad\<x,y,z\>_{1,s}=[y,z]_s\,,\quad\<x,y,z\>_{t,t}=[x,z]_t\,,
\end{equation*}
so the boundary of $\<\!\<x,y,z\>\!\>$ is composed of the three segments $[[x,y]]$\,, $[[y,z]]$ and $[[z,x]]$\,.

\medskip
A $1$-form $A$ on $\G$ will be seen as a (smooth) map $A:\G\rightarrow\g^*$ and it gives rise to a $1$-form $A\circ\exp:\g\rightarrow\g^*$. Its circulation through the segment $[[x,y]]$ is the real number
\begin{equation*}\label{circ}
\Gamma^A[[x,y]]\equiv\int_{[[x,y]]}\!A:=\int_0^1\!\big\<\log y-\log x\,\big\vert\, A\big([x,y]_s\big)\big\>\,ds\,.
\end{equation*}
It is easy to check that 
$$
\Gamma^0[[x,y]]=0\,,\quad\Gamma^A[[x,x]]=0\,,\quad\Gamma^A[[y,x]]=-\Gamma^A[[x,y]]\,.
$$

\begin{Remark}\label{global}
{\rm Setting $\,{\sf A}:=A\circ\exp:\g\rightarrow\g^*$, one gets a $1$-form on $\g$\,. Our definition is actually
\begin{equation*}\label{sari}
\Gamma^A[[x,y]]=\Gamma^{\sf A}[[\log x,\log y]]\,,
\end{equation*}
using the obvious and well-accepted definition of the circulation in the Lie algebra $\mathfrak g$ (a vector space)
\begin{equation*}\label{w-a}
\Gamma^{\sf A}[[X,Y]]:=\int_0^1\!\,\<\,Y-X\mid {\sf A}[(1-s)X+sY]\,\>\,ds\,.
\end{equation*}
}
\end{Remark}

\medskip
Let $B$ be a $2$-form on $\G$\,, seen as a smooth map sending $x\in\G$ into the skew-symmetric bilinear form $B(x):\g\times\g\rightarrow\mathbb R$\,. One defines
\begin{equation}\label{eguation}
\Gamma^B\<\!\<x,y,z\>\!\>\equiv\int_{\<\!\<x,y,z\>\!\>}\!\!B\,:=\int_0^1\!dt\int_0^t\!ds\,B\big(\<x,y,z\>_{t,s}\big)\big(\log x-\log y,\log x-\log z\big)\,.
\end{equation}

\begin{Proposition}\label{aya}
Let $B$ be a magnetic field, i.e. a closed $2$-form on $\G$\,. Then 
\begin{equation*}\label{asta}
\gamma^B(q;x,y):=e^{i\Gamma^B\<\!\<q,x^{-1}q,y^{-1}x^{-1}q\>\!\>}
\end{equation*}
defines a $2$-cocycle of $\,\G$ with values in $C(\G;\T)$\,, satisfying the extra condition
\begin{equation}\label{speciala}
\gamma^B(x^{-1}\!,x)=1=\gamma^B(x,x^{-1})\,,\quad\forall\,x\in\G\,,
\end{equation}
\end{Proposition}

\begin{proof}
Since $B$ has been supposed smooth, it is clear that $\gamma^B$ is well-defined and continuous.

By Stokes' Theorem, since $B$ is a closed $2$-form, one gets for $a,b,c,d\in\G$
\begin{equation*}\label{zdocs}
\Gamma^B\<\!\<a,b,c\>\!\>+\Gamma^B\<\!\<a,d,b\>\!\>+\Gamma^B\<\!\<b,d,c\>\!\>+\Gamma^B\<\!\<a,c,d\>\!\>=0\,.
\end{equation*}
Setting
$$
a:=q\,,\;b=x^{-1}q\,,\;c:=y^{-1}x^{-1}q\,,\; d=z^{-1}y^{-1}x^{-1}q\,,
$$
one gets
$$
\begin{aligned}
&\Gamma^B\<\!\<q,x^{-1}q,y^{-1}x^{-1}q\>\!\>+\Gamma^B\<\!\<q,z^{-1}y^{-1}x^{-1}q,x^{-1}q\>\!\>\\
+&\Gamma^B\<\!\<x^{-1}q,z^{-1}y^{-1}x^{-1}q,y^{-1}x^{-1}q\>\!\>+\Gamma^B\<\!\<q,y^{-1}x^{-1}q,z^{-1}y^{-1}x^{-1}q\>\!\>=0\,.
\end{aligned}
$$
Using the identity $\Gamma^B\<\!\<a,b,c\>\!\>=-\Gamma^B\<\!\<a,c,b\>\!\>$ for the second and the third terms and taking imaginary exponentials, one gets the $2$-cocycle identity. Normalization is easy, using \eqref{eguation} and the fact that every $B(z)$ is anti-symmetric; for example
$$
\gamma^B(q;x,\e)=e^{i\Gamma^B\<\!\<q,x^{-1}q,x^{-1}q\>\!\>}=1\,.
$$
The identity \eqref{speciala} reads $\,e^{i\Gamma^B\<\!\<q,xq,q\>\!\>}=1=e^{i\Gamma^B\<\!\<q,x^{-1}q,q\>\!\>},$ which is obvious from \eqref{eguation} since for every $z\in\G$ and $X\in\mathfrak g$ one has $[B(z)](X,0)=0$ \,.
\end{proof}

\begin{Corollary}\label{rolar}
 $\G$ possesses a twisted symmetric quantization.
\end{Corollary}

\begin{proof}
By \cite[Prop 4.3]{MR} a connected simply connected nilpotent Lie group admits a symmetric map $\tau$\,; it is given by 
$$
x\to\tau(x):=\int_0^1[\e,x]_s\,ds=\int_0^1\exp[s\log x]\,ds\,.
$$ 
Thus the result follows from Propositions \ref{symmetric} and Proposition \ref{aya} .
\end{proof}

Being a closed $2$-form, the magnetic field can be written as $B=dA$ for some $1$-form (vector potential). Any other vector potential $\tilde A$ satsfying $B=d\tilde A$ is related to the first by $\tilde A=A+d\psi$\,, where $\psi$ is a smooth function on $\G$\,.

\begin{Proposition}\label{ifort}
Suppose that $B=dA$\,. The relation 
$$
\beta^A(q;x):=e^{i\Gamma^A[\![q;x^{-1}q]\!]}\,,\quad x,q\in\G
$$ 
defines a pseudo-trivialization of the $2$-cocycle $\gamma^B$. 
\end{Proposition}

\begin{proof}
The identity $\gamma^B=\delta^1\big(\beta^A\big)$ is reduced to
\begin{equation}\label{circuflux}
\Gamma^B\<\!\<q,x^{-1}q,y^{-1}x^{-1}q\>\!\>=\Gamma^A[[q,x^{-1}q]]+\Gamma^A[[x^{-1}q,y^{-1}x^{-1}q]]+\Gamma^A[[q,(xy)^{-1}q]]\,,
\end{equation}
which follows from Stokes' Theorem.
\end{proof}

So we can write down all the formulas and results of the preceding subsections for the $2$-cocycle $\gamma^B$ and its pseudo-trivialization $\beta^A$. For example, writing $\mathbf{Op}_{\beta^A}\equiv\mathbf{Op}_{A}$ (a magnetic pseudo-differential operator) and  $\mathbf U_{\beta^A}\!\equiv\mathbf U_{A}$ (magnetic translations), one has
$$
[\mathbf{Op}_{A}(\mathbf s)u](x)=\int_{\G}\int_{\mathfrak g^*} e^{i\int_{[[x,y]]}\!A}\,e^{i\<\log(xy^{-1})\mid\mathcal X\>}\mathbf s\big(x,\mathcal X\big)u(y)\,d\m(y)d\mathcal X\,,
$$
$$
\big[\mathbf U_{A}(x)u\big](q)=e^{i\int_{[[q,x^{-1}q]]}\!A}\,u(x^{-1}q)=e^{i\int_0^1\<\log(x^{-1}q)-\log q\,\mid A([q,x^{-1}q]_s)\>ds}\,u(x^{-1}q) \,.
$$

If $\G=\mathbb{R}^n$, the dual group can be identified the vector space dual $\mathbb{R}^n$. It is also true that $\R^n$ is identified with its Lie algebra and (then) with its dual, so in this case the maps $\exp$ and $\log$ simply disappear from the formulas. In this case, if $\tau(x)=x/2$\,, one get the symetric (Weyl) twisted quantization.

\bigskip
\noindent
{\bf Acknowledgements:} The authors have been supported by the N\'ucleo Milenio de F\'isica Matem\'atica RC120002. M. M. acknowledges support from the Fondecyt Project 1160359.


\end{document}